\documentclass{siamltex}
\usepackage{lineno,hyperref}
\usepackage{multirow}
\usepackage[english]{babel}
\usepackage{amsmath}
\usepackage{bm}
\usepackage[matha,mathx]{mathabx}
\usepackage{mathrsfs}
\usepackage{calligra}
\usepackage{calrsfs}
\usepackage{amsfonts}
\usepackage{amssymb}
\usepackage{graphicx}
\usepackage{tikz}
\usepackage{here}
\usepackage{algorithm}
\usepackage{algorithmic}
\usepackage{graphicx}
\usepackage{amsmath}
\usepackage{here}
\usetikzlibrary{matrix,calc}
\usepackage{caption}
\newcommand{\R}{\mathbb{R}}
\newcommand{\C}{\mathbb{C}}
\newcommand{\K}{\mathbb{K}}

\newtheorem{lemm}{Lemma}

\newtheorem{theo}{Theorem}

\newtheorem{defi}[theo]{Definition}

\title{A tensor bidiagonalization method for higher-order singular value 
	decomposition with applications}

\author{A. El Hachimi\footnotemark[2] \thanks{Laboratory MSDA, Mohammed VI Polytechnic University, Green 
		City, Morocco.} 
	\and 
	K. Jbilou\footnotemark[1] \thanks{Universit\'e du Littoral Cote d'Opale, LMPA, 50 rue F. Buisson, 62228 
			Calais-Cedex, France}
	\and A. Ratnani\footnotemark[1]
	\and L. Reichel\thanks{Department of Mathematical Sciences, Kent State University, Kent, 
		OH 44242, USA.}
}

\begin{document}
	\maketitle
	
	\begin{abstract}
		The need to know a few singular triplets associated with the largest singular values of 
		third-order tensors arises in data compression and extraction. This paper describes a 
		new method for their computation using the t-product. Methods for determining a couple of
		singular triplets associated with the smallest singular values also are presented. The 
		proposed methods generalize available restarted Lanczos bidiagonalization methods for 
		computing a few of the largest or smallest singular triplets of a matrix. The methods of 
		this paper use Ritz and harmonic Ritz lateral slices to determine accurate 
		approximations of the largest and smallest singular triplets, respectively. Computed 
		examples show applications to data compression and face recognition.
	\end{abstract}
	

	\section{Introduction}
	The last 20 years has seen an immense growth of the amount of data that is collected for 
	analysis, but it is a challenging problem to extract useful information from available 
	data. This difficulty arises, e.g., in machine learning, data mining, and deep learning; 
	see, e.g., Arnold et al. \cite{arnold2018}. The extraction of useful information from data
	that is represented by a \emph{matrix} often is facilitated by the singular value 
	decomposition of the matrix. Typically, only a few of the largest singular triplets, i.e.,
	the largest singular values and associated right and left singular vectors, are required 
	to extract useful information from the matrix. A restarted Lanczos bidiagonalization 
	method for computing accurate approximations of these singular triplets is described in 
	\cite{baglama2005augmented}, and R code written by Bryan W. Lewis is available at
	\cite{lewis2021}.
	
	In many recent applications the given data are represented by a multidimensional array. 
	These arrays, known as \emph{tensors}, are natural generalizations of matrices. Several
	approaches to define tensor-tensor products and tensor-matrix products are described in
	the literature, including the $n$-mode product \cite{BNR,kolda2009tensor}, the t-product 
	\cite{kilmer2013third,RU22}, and the c-product \cite{kernfeld2015tensor,RU21}. Generalizations of 
	the singular value decomposition (SVD) to tensors are described in \cite{kolda2009tensor} using
	the $n$-mode product (the so-called HOSVD), and in \cite{kernfeld2015tensor,kilmer2013third} 
	using the tensor c-product and 
	t-product. The need to compute the SVD or a partial SVD of a tensor arises in a variety
	of applications, including image restoration,  tensor completion \cite{elha1}, 
	robust tensor principal 
	component analysis \cite{lu2019tensor}, tensor compression \cite{kilmer-compression}, and 
	recognition of color faces \cite{hached2021multidimensional,hao2013facial}. These 
	applications require knowledge of the largest singular values and associated lateral
	tensor singular slices.
	
	It is the purpose of the this paper to introduce a new restarted tensor Lanczos 
	bidiagonalization method for third-order tensors using the t-product for approximating 
	a few of the largest singular values and associated lateral tensor singular slices. This 
	method generalizes the approach described in \cite{baglama2005augmented} from matrices to 
	tensors. We remark that the Lanczos bidiagonalization method (also known as the 
	Golub-Kahan bidiagonalization method) for third-order tensors using the t-product has
	been described in \cite{Elguide,ElIchi,kilmer2013third,RU2022}; however, this bidiagonalization method
	differs from the one of the present paper.
	
	In \cite{baglama2005augmented} the authors also describe a restarted Lanczos bidiagonalization method
	for the computation of a few of the smallest singular values and associated singular 
	vectors of a large matrix by determining harmonic Ritz values is presented. This paper 
	presents an analogous scheme for third-order tensors. 
	
	The organization of this paper is as follows. Section \ref{sec2} recalls some properties 
	of the t-product and Section \ref{sec4} reviews tensor Lanczos bidiagonalization of 
	third-order tensors using the t-product. Restarted tensor Lanczos bidiagonalization 
	methods are presented for the approximation of a few of the largest singular values and 
	associated lateral tensor singular slices by computing lateral tensor Ritz slices, as well
	as for approximating a few of the smallest singular values and associated lateral tensor 
	singular slices by evaluating harmonic lateral tensor Ritz slices. Section \ref{sec5} 
	discusses multidimensional principal component analysis using a partial tensor HOSVD with 
	application to face recognition, and Section \ref{sec6} presents a few computed examples. 
	Concluding remarks and possible extensions can be found in Section \ref{sec7}.

	\section{The tensor t-product}\label{sec2}
	This section reviews results by Kilmer et al. \cite{kilmer2013third,kilmer} and uses 
	notation employed there and by Kolda and Bader \cite{kolda2009tensor}. A third-order 
	tensor is an array $\mathcal{A}=[a_{ijk}]\in\R^{\ell\times p\times n}$. Matrices and 
	vectors are tensors of order two and one, respectively. A \emph{slice} or \emph{frame} of
	a third-order tensor $\mathcal{A}$ is a section obtained by fixing any one of the three 
	indices. Using MATLAB notation, $\mathcal{A}(i,:,:)$, $\mathcal{A}(:,j,:)$, and 
	$\mathcal{A}(:,:,k)$ denote the $i$th horizontal, the $j$th lateral, and the $k$th frontal
	slices of $\mathcal{A}$, respectively. The lateral slice $\mathcal{A}(:,j,:)$ also is 
	denoted by $\vec{\mathcal A}_j$, and the frontal slice $\mathcal{A}(:,:,k)$ is an 
	$\ell\times p$ matrix that is sometimes denoted by $\mathcal{A}^{(k)}$. A \emph{fiber} of a third 
	order tensor $\mathcal{A}$ is defined by fixing any two of the three indices. The fiber 
	$\mathcal{A}(i,j,:)$ is called a \emph{tube} of $\mathcal{A}$. We will use capital 
	calligraphic letters $\mathcal{A}$ to denote third-order tensors, capital letters $A$ to 
	identify matrices, bold face lower case letters $\bm{a}$ to denote tubes, and lower case 
	letters $a$ stand for scalars. Further, $\K^{\ell\times p}_n=\R^{\ell\times p\times n}$ 
	denotes the space of third-order tensors of size $\ell\times p\times n$, 
	$\K^\ell_n=\R^{\ell\times 1\times n}$ stands for the space of lateral slices of size 
	$\ell\times n$, and $\K_n=\R^{1\times 1\times n}$ denotes the space of tubes with $n$ 
	entries. For a third-order tensor $\mathcal{A}\in\K^{\ell\times p}_n$ with frontal slices 
	$\mathcal{A}^{(i)}$, $i=1,\ldots,n$, we define:
	\begin{itemize}
		\item The block circulant matrix associated with $\mathcal{A}$: 
		\begin{equation}\label{bcirc}
			{\tt bcirc}(\mathcal{A})=\begin{bmatrix}
				\mathcal{A}^{(1)}&  \mathcal{A}^{(n)} & \ldots & \mathcal{A}^{(2)}\\
				\mathcal{A}^{(2)}&  \mathcal{A}^{(1)} & \ldots & \mathcal{A}^{(3)}\\
				\vdots & \ddots & \ddots & \vdots \\
				\mathcal{A}^{(n)}&  \mathcal{A}^{(n-1)} & \ldots & \mathcal{A}^{(1)}\\
			\end{bmatrix}\in \K^{\ell n \times pn}.
		\end{equation}
		\item The operator ${\tt unfold}$ applied to ${\mathcal A}$ gives the matrix made up of
		its frontal slices,
		\[
		{\tt unfold}(\mathcal{A})=\begin{bmatrix}
			\mathcal{A}^{(1)}\\
			\mathcal{A}^{(2)}\\
			\vdots \\
			\mathcal{A}^{(n)}
		\end{bmatrix} \in \K^{\ell n\times p}.
		\]
		We also will need the inverse operator ${\tt fold}$ such that 
		${\tt fold}\left(\tt unfold\left(\mathcal{A}\right)\right)=\mathcal{A}$.
		\item The block diagonal matrix associated with $\mathcal{A}$ is defined as
		\[{\tt bdiag}\left(\mathcal{A}\right)=\begin{bmatrix}
			\mathcal{A}^{(1)} &    &    &  \\
			& \mathcal{A}^{(2)} &   &  \\
			&  &  \ddots &  \\
			&  &  &  \mathcal{A}^{(n)}
		\end{bmatrix}\in \K^{\ell n\times pn}.
		\]
	\end{itemize}
	
	\begin{defi}\label{def1}(\cite{kilmer})
		Let $\mathcal{A}\in\K^{\ell\times q}_{n}$ and $\mathcal{B}\in\K^{q\times p}_{n}$ be 
		third-order tensors. The t-product of $\mathcal{A}$ and $\mathcal{B}$ is defined by 
		\[
		\mathcal{A}\star\mathcal{B} := {\tt fold}\left({\tt bcirc}(\mathcal{A})\, 
		{\tt unfold}(\mathcal{B})\right)\in\K^{\ell\times p}_n.
		\]
	\end{defi}
	
	The block circulant matrix \eqref{bcirc} can be block-diagonalized by using the discrete 
	Fourier transform (DFT) as follows:
	\[
	{\tt bcirc}(\mathcal{A})=\left(F_n^H\otimes I_\ell\right)
	{\tt bdiag}(\widehat{\mathcal{A}})
	\left(F_n \otimes I_p\right),
	\]
	where $F_n\in\C^{n \times n}$ is the discrete Fourier matrix, $F_n^H$ denotes its 
	conjugate transpose, $\widehat{\mathcal{A}}$ stands for the Fourier transform of 
	$\mathcal{A}$ along each tube, $I_\ell\in\R^{\ell\times\ell}$ denotes the identity 
	matrix, and $\otimes$ is the Kronecker product. The matrix $\widehat{\mathcal{A}}$ can be
	computed with the fast Fourier transform (FFT) algorithm; see \cite{kilmer} for details. 
	Using MATLAB notations, we have 
	\[
	\widehat{\mathcal{A}}={\tt fft}(\mathcal{A}, [\, ], 3).
	\]
	The inverse operation can be evaluated in MATLAB with the command
	\[
	\mathcal{A}={\tt ifft}(\widehat{\mathcal{A}},[\,],3).
	\]
	Hence, the t-product $\mathcal{C}=\mathcal{A}\star\mathcal{B}$ can be evaluated as
	\begin{equation} \label{eq p}
		\widehat{\mathcal{C}}^{(i)}=\widehat{\mathcal{A}}^{(i)} \widehat{\mathcal{B}}^{(i)},
		\qquad i=1,2,\ldots,n,
	\end{equation}
	where $\widehat{\mathcal{A}}^{(i)}$, $\widehat{\mathcal{B}}^{(i)}$, and 
	$\widehat{\mathcal{C}}^{(i)}$ are the $i$th frontal slices of the tensors 
	$\widehat{\mathcal{A}}$, $\widehat{\mathcal{B}}$, and $\widehat{\mathcal{C}}$, 
	respectively. 
	
	As already pointed out by Kilmer et al. \cite{kilmer2013third}, one can use symmetry 
	properties of the DFT when applied to real data to reduce the computational effort when 
	evaluating the t-product with the FFT. This is described by the following result, which
	can be found, e.g., in \cite{Rojo}.
	
	\begin{lemm}
		Given a real vector $v\in\R^{n}$, the associated DFT vector $\widehat{v}=F_n v$ satisfies
		\[
		\widehat{v}_1\in\R,\quad {\tt conj}\left(\widehat{v}_i\right)=
		\widehat{v}_{n-i+2},\quad i=2,3,\ldots,\left[\dfrac{n+1}{2}\right],
		\]
		where ${\tt conj}$ denotes the complex conjugation operator and 
		$\left[\dfrac{n+1}{2}\right]$ denotes the integer part of $\dfrac{n+1}{2}$.
	\end{lemm}
	
	It follows that for a third-order tensor $\mathcal{A}\in\K^{\ell\times p}_{n}$, we have
	\[
	\widehat{\mathcal{A}}^{(1)}\in\R^{\ell\times p}, \quad
	{\tt conj}\left(\widehat{\mathcal{A}}^{(i)}\right)=
	\widehat{\mathcal{A}}^{(n-i+2)},\quad i=2,3,\ldots,\left[\dfrac{n+1}{2}\right].
	\]
	This shows that the t-product of two third-order tensors can be determined by evaluating 
	just about half the number of products involved in \eqref{eq p}. Algorithm \ref{alg:1}
	describes the computations.
	
	\begin{algorithm}[H]
		\caption{t-product of third-order tensors.}\label{alg:1}
		\textbf{Input}: $\mathcal{A}\in\K^{\ell\times q}_{n}$, $\mathcal{B}\in\K^{q\times p}_{n}$. \\
		\textbf{Output}: $\mathcal{C}:=\mathcal{A}\star \mathcal{B}\in\K^{\ell\times p}_{n}$.
		\begin{algorithmic}[1]
			\STATE Compute $\widehat{\mathcal{A}}={\tt fft}(\mathcal{A}, [\,], 3)$, 
			$\widehat{\mathcal{B}}={\tt fft}(\mathcal{B}, [\,], 3)$.
			\FOR  {$i=1,\ldots,  \left[\dfrac{n+1}{2}\right]$}
			\STATE  $\phantom{000}\widehat{\mathcal{C}}^{(i)}=
			\widehat{\mathcal{A}}^{(i)}\widehat{\mathcal{B}}^{(i)}$.
			\ENDFOR
			\FOR {$i=\left[\dfrac{n+1}{2}\right]+1 , \ldots, n$}
			\STATE $\phantom{000}\widehat{\mathcal{C}}^{(i)}=
			{\tt conj}\left(\widehat{\mathcal{C}}^{(n-i+2)}\right)$.
			\ENDFOR
			\STATE $\mathcal{C}={\tt ifft} \left(\widehat{\mathcal{C}},[\,],3\right)$.
		\end{algorithmic}
	\end{algorithm}
	
	The following definition is concerned with the t-product of a third-order tensor and a 
	tube.
	
	\begin{defi}
		Let $\mathcal{A}\in\K^{\ell\times p}_n$ and $\bm{b}\in\K_n$. Then 
		$\mathcal{C}:=\mathcal{A}\star\bm{b}\in\K^{\ell\times p}_n$ is obtained by applying the 
		inverse DFT along each tube of $\widehat{\mathcal{C}}$, where each frontal slice is 
		determined by the standard matrix product between each frame of $\widehat{\mathcal{A}}$ 
		and $\widehat{\bm{b}}$, i.e.,
		\[
		\widehat{\mathcal{C}}^{(i)}=\widehat{\mathcal{A}}^{(i)}\widehat{\bm{b}}^{(i)}=
		\textcolor{black}{\widehat{\bm{b}}^{(i)}\widehat{\mathcal{A}}^{(i)}},\quad i=1,2,\ldots,n.
		\]
	\end{defi}
	
	A third-order tensor $\mathcal{A}\in\K^{\ell\times p}_n$ can be written as
	\[
	\mathcal{A}=\left[\vec{\mathcal{A}}_1,\vec{\mathcal{A}}_2,\ldots,\vec{\mathcal{A}}_p
	\right],
	\]
	thus, for the tensors $\mathcal{A}\in\K^{\ell\times q}_n$ and 
	$\mathcal{B}\in\K^{q\times p}_n$, the t-product $\mathcal{A}\star \mathcal{B}$ can be 
	expressed as 
	\[
	\mathcal{A}\star\mathcal{B}=\left[\mathcal{A}\star\vec{\mathcal{B}}_1,\mathcal{A}\star
	\vec{\mathcal{B}}_2,\ldots,\mathcal{A}\star\vec{\mathcal{B}}_p\right], 
	\]
	where
	\[
	\mathcal{A}\star\overrightarrow{\mathcal{B}}_i=
	\overrightarrow{\left(\mathcal{A}\star\mathcal{B}\right)}_i, \quad i=1,2,\ldots,p.
	\]
	
	The Frobenius norm of a third-order tensor $\mathcal{A}\in\K^{\ell\times p}_n$ is given by 
	\[
	\left\Vert\mathcal{A}\right\Vert_F:=\sqrt{\sum_{i_1,i_2,i_3=1}^{\ell,p,n} 
		a_{i_1,i_2,i_3}^2},
	\]
	and the inner product of two third-order tensors of the same size
	$\mathcal{A},\mathcal{B}\in\K^{\ell\times p}_{n}$ is defined as
	\[
	\langle\mathcal{A},\mathcal{B}\rangle:=\sum_{i_1,i_2,i_3=1}^{\ell,p,n}
	a_{i_1,i_2,i_3}b_{i_1,i_2,i_3}. 
	\]
	We have the relations
	\[
	\left\Vert\mathcal{A}\right\Vert_F=\dfrac{1}{\sqrt{n}}\left\Vert\widehat{\mathcal{A}}
	\right\Vert_F,\qquad
	\langle\mathcal{A},\mathcal{B}\rangle=\dfrac{1}{n}\langle\widehat{\mathcal{A}},
	\widehat{\mathcal{B}}\rangle.
	\]
	
	We recall for later use the definitions of some special tensors and operations:
	\begin{itemize}
		\item The identity tensor $\mathcal{I}_\ell\in \K^{\ell\times\ell}_n$ is the tensor whose 
		first frontal slice is the identity matrix and all other slices have zero entries only. 
		\item The transpose of a real third-order tensor, $\mathcal{A}\in\K^{\ell\times p}_n$,
		denoted by $\mathcal{A}^H\in\K^{p\times\ell}_n$, is the tensor obtained by first 
		transposing each one of the frontal slices of $\mathcal{A}$, and then reversing the order
		of the transposed frontal slices $2$ through $n$; see \cite{kilmer}. Let the third-order
		tensors ${\mathcal A}$ and $\mathcal{B}$ be such that the products 
		${\mathcal A}\star\mathcal{B}$ and $\mathcal{B}^H\star\mathcal{A}^H$ are defined. Then, 
		similarly to the matrix transpose, the tensor transpose satisfies 
		$({\mathcal A}\star\mathcal{B})^H=\mathcal{B}^H\star\mathcal{A}^H$.
		\item A tensor $\mathcal{Q}\in\K^{\ell\times\ell}_n$ is said to be orthogonal if 
		and only if 
		\[
		\mathcal{Q}^H\star\mathcal{Q}=\mathcal{Q}\star\mathcal{Q}^H=\mathcal{I}_\ell.
		\]
		\item A square third-order tensor $\mathcal{A}\in\K^{\ell\times\ell}_n$ is invertible if 
		there is a third-order tensor $\mathcal{B}\in\K^{\ell\times\ell}_n$ such that
		\[
		\mathcal{A}\star\mathcal{B}=\mathcal{I}_\ell,\quad 
		\mathcal{B}\star\mathcal{A}=\mathcal{I}_\ell.
		\]
		In this case $\mathcal{B}$ is said to be the inverse of $\mathcal{A}$, and is denoted by
		$\mathcal{A}^{-1}$.
	\end{itemize}
	
	\begin{defi}(\cite{kilmer2013third})
		Let $\vec{\mathcal{A}}_i\in\K^{\ell}_n$ for $i=1,2,\ldots,p$ be lateral slices of the
		tensor $\mathcal{A}\in\K_n^{\ell\times p}$. A t-linear combination of these slices
		is defined as
		\[
		\vec{\mathcal{A}}_1\star\bm{b}_1+\vec{\mathcal{A}}_2\star\bm{b}_2+
		\ldots+\vec{\mathcal{A}}_p\star\bm{b}_p,  
		\]
		where the $\bm{b}_i$ for $i=1,2,\ldots,p$ are tubes in $\K_n$. Moreover,	
		\[
		{\tt span}\left\{\vec{\mathcal{A}}_1,\vec{\mathcal{A}}_2,\ldots,\vec{\mathcal{A}}_p\right\}=
		\left\{\sum_{i=1}^{p}\vec{\mathcal{A}}_i\star\bm{b}_i :~~ \bm{b}_i\in\K_n,
		~~i=1,2,\ldots,p\right\}.
		\]
	\end{defi}
	
	The tensor singular value decomposition (t-SVD) associated with the t-product, introduced
	by Kilmer and Martin \cite{kilmer}, generalizes the classical SVD of a matrix. It is 
	described in the next theorem.
	
	\begin{theo}(\cite{kilmer})
		Let $\mathcal{A}\in\K^{\ell\times p}_{n}$ be a third-order tensor. Then it can be 
		represented as the t-product of three third-order tensors,
		\begin{equation}\label{tsvd}
			\mathcal{A}=\mathcal{U}\star \mathcal{S}\star \mathcal{V}^H,
		\end{equation}
		where $\mathcal{U}\in\K^{\ell\times\ell}_n$ and $\mathcal{V}\in\K^{p\times p}_n$ are 
		orthogonal tensors, and $\mathcal{S}\in\K^{\ell\times p}_n$ is an f-diagonal tensor, i.e., each frontal slice  of the DFT of $\mathcal{S}$ is a diagonal matrix.
	\end{theo}
	
	Algorithm \ref{alg:2} summarizes the computation of the t-SVD of a third-order tensor with
	the aid of the FFT. 
	
	\begin{algorithm}[H]
		\caption{The t-SVD of a third-order tensor.}\label{alg:2}
		\textbf{Input}: $\mathcal{A}\in\K^{\ell\times p}_{n}$. \\
		\textbf{Output}: $\mathcal{U}\in\K^{\ell\times\ell}_n,\; 
		\mathcal{S}\in\K^{\ell\times p}_n,\; \mathcal{V}\in\K^{p\times p}_n$.
		\begin{algorithmic}[1]
			\STATE $\widehat{\mathcal{A}}={\tt fft}(\mathcal{A}, [\,], 3).$
			\FOR {$i=1, \ldots,\left[\dfrac{n+1}{2}\right]$}
			\STATE $\phantom{000}[\widehat{\mathcal{U}}^{(i)},\widehat{\mathcal{S}}^{(i)},\widehat{\mathcal{V}}^{(i)}]=
			{\tt svd}(\widehat{\mathcal{A}}^{(i)})$.
			\ENDFOR
			\FOR{$i=1,\ldots,\left[\dfrac{n+1}{2}\right]+1$}
			\STATE $\phantom{000}\widehat{\mathcal{U}}^{(i)}=
			{\tt conj}\left(\widehat{\mathcal{U}}^{(n-i+2)}\right)$, 
			$\widehat{\mathcal{S}}^{(i)}={\tt conj}\left(\widehat{\mathcal{S}}^{(n-i+2)}\right)$, and 
			$\widehat{\mathcal{V}}^{(i)}={\tt conj}\left(\widehat{\mathcal{V}}^{(n-i+2)}\right)$.
			\ENDFOR
			\STATE Compute $\mathcal{U}={\tt ifft}(\widehat{\mathcal{U}},[\,], 3)$, 
			$\mathcal{S}={\tt ifft}(\widehat{\mathcal{S}},[\,], 3)$, and 
			$\mathcal{V}={\tt ifft}(\widehat{\mathcal{V}},[\,], 3)$.
		\end{algorithmic}
	\end{algorithm}
	
	The factorization \eqref{tsvd} can be expressed as 
	\[
	\mathcal{A}=\mathcal{U}\star\mathcal{S}\star\mathcal{V}^H=
	\sum_{i=1}^{\min\{\ell,p\}}\vec{\mathcal{U}}_i\star\bm{s}_i\star\vec{\mathcal{V}}_i^H, 
	\]
	where the $\bm{s}_i=\mathcal{S}(i,i,:)$ are singular tubes, and 
	$\vec{\mathcal{U}}_i=\mathcal{U}(:,i,:)$ and $\vec{\mathcal{V}}_i=\mathcal{U}(:,i,:)$ are
	right and left lateral tensor singular slices, respectively, for 
	$i=1,2,\ldots,\min(\ell,p)$. The triplets 
	$\{\bm{s}_i,\vec{\mathcal{U}}_i,\vec{\mathcal{V}}_i\}_{i=1:\min(\ell,p)}$ will be 
	referred to as singular triplets of the tensor $\mathcal{A}$. The singular tubes are 
	ordered so that their norms $\sigma_i=\Vert\bm{s}_i\Vert_F$ are decreasing with $i$, i.e., 
	\[
	\sigma_1\geq\sigma_2\geq\ldots\geq\sigma_{\min(\ell,p)} \geq 0.
	\]
	Note that we also have the relations 
	\[
	\mathcal{A}\star\vec{\mathcal{V}}_i=\vec{\mathcal{U}}_i\star\bm{s}_i,\quad 
	\mathcal{A}^H\star\vec{\mathcal{U}}_i=\vec{\mathcal{V}}_i\star\bm{s}_i,\quad
	i=1,2,\ldots,\min\{\ell,p\}.
	\]
	We remark that the latter relations have to be modified if ${\mathcal A}$ has 
	complex-valued entries.
	
	We note for future reference that
	\begin{equation}\label{new 2.21}
		\mathcal{S}(i,i,1)=\sum_{j=1}^{n}\dfrac{1}{n}\widehat{\mathcal{S}}(i,i,j).
	\end{equation}
	
	In the following, we will need the notion of rank of a third-order tensor. 
	
	\begin{defi}
		Let $\mathcal{A}\in\K^{\ell\times p}_n$ be a third-order tensor. Then its tubal rank is 
		defined as 
		\[
		{\tt rank}_t\left(\mathcal{A}\right)={\tt card}\left\{\sigma_i\neq 0,~~
		i=1,2,\ldots,\min\{\ell,p\} \right\},
		\]
		where $\sigma_i$ is the norm of the singular tube $\bm{s}_i$ of $\mathcal{A}$ and
		${\tt card}$ stands for the cardinality.
	\end{defi}
	
	The next result generalizes the Eckart-Young theorem for matrices to third-order tensors.
	It is important in the context of data compression.
	
	\begin{theo}(\cite{kilmer-compression,kilmer})
		Let the t-SVD of a third-order tensor $\mathcal{A}\in\K^{\ell\times p}_n$ be given by
		$\mathcal{A}=\mathcal{U}\star\mathcal{S}\star \mathcal{V}^H$. For 
		$1\leq k\leq\min\{\ell,p\}$, define the truncated t-SVD by 
		\begin{equation*}
			\mathcal{A}_k=\sum_{i=1}^{k}\vec{\mathcal{U}}_i\star\bm{s}_i\star\vec{\mathcal{V}}_i^H.
		\end{equation*}
		Then 
		\[
		\mathcal{A}_k=\underset{\widetilde{\mathcal{A}}\in\mathbb{M}}{\arg\min}
		\left\Vert\mathcal{A}-\widetilde{\mathcal{A}}\right\Vert_F.
		\]
		Where $\mathbb{M}$ is the set given by $\mathbb{M}=\{\mathcal{X}\star \mathcal{Y}; \; \text{with}\; \mathcal{X}\in \K^{l\times k}_n,\; \mathcal{Y}\in \K^{k\times p}_n \}$.
	\end{theo}
	
	The matrix QR factorization also can be generalized to tensors.
	
	\begin{theo}(\cite{kilmer})
		Let $\mathcal{A}\in\K^{\ell\times p}_n$. Then $\mathcal{A}$ can be factored as 
		\begin{equation}\label{tQR}
			\mathcal{A}=\mathcal{Q}\star \mathcal{R},
		\end{equation}
		where $\mathcal{Q}\in\K^{\ell\times\ell}_n$ is an orthogonal tensor and 
		$\mathcal{R}\in\K^{\ell\times p}_n$ is an f-upper triangular tensor, i.e., each frontal 
		slice of the DFT of $\mathcal{R}$ is an upper triangular matrix. The factorization 
		\eqref{tQR} is referred to as the t-QR factorization of $\mathcal{A}$.
	\end{theo}
	
	Algorithm \ref{alg:3} summarizes the computation of the t-QR factorization \eqref{tQR}.
	The function ${\tt qr}$ in line 3 of the algorithm computes a QR factorization of the
	matrix $\widehat{\mathcal{A}}^{(i)}\in\R^{\ell\times p}$; thus 
	$\widehat{\mathcal{A}}^{(i)}=\widehat{\mathcal{Q}}^{(i)}\widehat{\mathcal{R}}^{(i)}$, 
	where the matrix $\widehat{\mathcal{Q}}^{(i)}\in\R^{\ell\times\ell}$ is orthogonal and 
	the matrix $\widehat{\mathcal{R}}^{(i)}\in\R^{\ell\times p}$ has an upper triangular
	leading principal submatrix of order $\ell$.
	
	\begin{algorithm}[H]
		\caption{t-QR factorization of a third-order tensor.}\label{alg:3}
		\textbf{Input}: $\mathcal{A}\in\K^{\ell\times p}_n$. \\
		\textbf{Output}: $\mathcal{Q}\in\K^{\ell\times\ell}_n,\; 
		\mathcal{R}\in\K^{\ell\times p}_n$.
		\begin{algorithmic}[1]
			\STATE $\widehat{\mathcal{A}}={\tt fft}(\mathcal{A}, [\,], 3).$
			\FOR  {$i=1\, \ldots, \left[\dfrac{n+1}{2}\right]$}
			\STATE $\phantom{000}[\widehat{\mathcal{Q}}^{(i)},
			\widehat{\mathcal{R}}^{(i)}]={\tt qr}(\widehat{\mathcal{A}}^{(i)})$.
			\ENDFOR
			\FOR {$i=\left[\dfrac{n+1}{2}\right]+1 \ldots,n$}
			\STATE $\phantom{000}\widehat{\mathcal{Q}}^{(i)}={\tt conj}
			\left(\widehat{\mathcal{Q}}^{(n-i+2)}\right)$ and  
			$\widehat{\mathcal{R}}^{(i)}={\tt conj}\left(\widehat{\mathcal{R}}^{(n-i+2)}\right)$.
			\ENDFOR
			\STATE Compute $\mathcal{Q}={\tt ifft}(\widehat{\mathcal{Q}},[\,], 3)$ and  
			$\mathcal{R}={\tt ifft}(\widehat{\mathcal{R}},[\,], 3)$.
		\end{algorithmic}
	\end{algorithm}
	
	Following Kilmer et al. \cite{kilmer2013third}, we define orthogonality of lateral tensor 
	slices. Let $\vec{\mathcal{X}}$ and $\vec{\mathcal{Y}}$ be two lateral 
	tensor slices in $\K^{\ell}_{n}$ and define the inner product of these slices as
	\[
	\left<\vec{\mathcal{X}},\vec{\mathcal{Y}}\right>:=\vec{\mathcal{X}}^H\star
	\vec{\mathcal{Y}}\in\K_{n}.
	\]
	The lateral slices in the set 
	\begin{equation}\label{setslices}
		\left\{\vec{\mathcal{X}}_1,\vec{\mathcal{X}}_2,\ldots,\vec{\mathcal{X}}_p\right\}, 
	\end{equation}
	with $p\geq 2$, are said to be \emph{orthogonal} if 
	\[
	\left<\vec{\mathcal{X}}_i, \vec{\mathcal{X}}_j\right>=\left\lbrace
	\begin{array}{cc}
		\alpha_i \textbf{e}_1 & \mbox{if $i=j$},\\
		\textbf{0} & \mbox{if $i\ne j$},
	\end{array}
	\right.
	\]
	where $\textbf{e}_1$ is the tube in $\K_n$, whose its first element is $1$ and the 
	remaining elements vanish, and the $\alpha_i$, $i=1,2,\ldots,p$, are nonvanishing scalars.
	Furthermore, if $\alpha_i=1$ for all $i=1,2,\ldots,p$, then the set \eqref{setslices} is 
	said to be \emph{orthonormal}.
	
	Following \cite{kilmer2013third}, we observe that any lateral slice 
	$\vec{\mathcal{X}}\in\K^{\ell}_{n}$ can be normalized as 
	\begin{equation}\label{normalx}
		\vec{\mathcal{X}}=\vec{\mathcal{Y}}\star\textbf{a}
	\end{equation}
	with $\vec{\mathcal{Y}}\in\K^{\ell}_{n}$, $\left\Vert\vec{\mathcal{Y}}\right\Vert=1$, and
	$\textbf{a}\in\K_n$. Here the tensor norm is defined as
	\[
	\left\Vert\vec{\mathcal{Y}}\right\Vert=
	\dfrac{\left\Vert\left<\vec{\mathcal{Y}},\vec{\mathcal{Y}}\right>\right\Vert_F}
	{\left\Vert \vec{\mathcal{Y}}\right\Vert_F}.
	\]
	Note that $\vec{\mathcal{Y}}$ has unit norm if and only if 
	$\left<\vec{\mathcal{Y}},\vec{\mathcal{Y}}\right>=\bm{e}_1$; see \cite{kilmer2013third} for 
	more detail. Algorithm \ref{alg:4} summarizes the normalization process. The MATLAB 
	function ${\tt randn}$ in the algorithm generates a vector in $\R^\ell$ with normally
	distributed pseudorandom entries with mean zero and variance one.
	
	\begin{algorithm}[H]
		\caption{{\tt Normalize}($\vec{\mathcal{X}}$).}\label{alg:4}
		\textbf{Input}: $\vec{\mathcal{X}}\in \mathbb{K}^{\ell}_{n}$. \\
		\textbf{Output}: $\vec{\mathcal{Y}}\in \mathbb{K}^{\ell}_{n}$ of unit norm and 
		$\textbf{a}\in\K_n$ that satisfy \eqref{normalx}.\\
		\begin{algorithmic}[1]
			\STATE $\vec{\widehat{\mathcal{Y}}}={\tt fft}(\vec{\mathcal{X}}, [\,], 3)$.
			\FOR {$i=1,\ldots, \left[\dfrac{n+1}{2}\right]$}
			\STATE $\widehat{\textbf{a}}^{(i)}=\left\Vert\vec{\widehat{\mathcal{Y}}}^{(i)}\right\Vert_F$.
			\IF {$\widehat{\textbf{a}}^{(i)}> 0$}
			\STATE $\vec{\widehat{\mathcal{Y}}}^{(i)}=
			\dfrac{\vec{\widehat{\mathcal{Y}}}^{(i)}}{\widehat{\textbf{a}}^{(i)}}$ 
			\ELSE
			\STATE $\vec{\widehat{\mathcal{Y}}}^{(i)}={\tt randn}(\ell,1);\; \textbf{b}^{(i)}=
			\left\Vert \vec{\widehat{\mathcal{Y}}}^{(i)}\right\Vert_F$, and 
			$\vec{\widehat{\mathcal{Y}}}^{(i)}=\dfrac{\vec{\widehat{\mathcal{Y}}}^{(i)}}
			{\textbf{b}^{(i)}}$. 
			\ENDIF
			\ENDFOR
			\FOR {$i=\left[\dfrac{n+1}{2}\right]+1, \ldots, n$}
			\STATE  $\vec{\widehat{\mathcal{Y}}}^{(i)}=
			{\tt conj}\left(\vec{\widehat{\mathcal{Y}}}^{(n-i+2)}\right)$, 
			$\widehat{\textbf{a}}^{(i)}={\tt conj}\left(\widehat{\textbf{a}}^{(n-i+2)}\right)$.
			\ENDFOR
			\STATE  $\vec{\widehat{\mathcal{Y}}}={\tt ifft}(\vec{\widehat{\mathcal{Y}}},[\,],3)$,
			$\textbf{a}={\tt ifft}(\widehat{\textbf{a}},[\,],3)$.
		\end{algorithmic}
	\end{algorithm}

	\section{Tensor Lanczos bidiagonalization for computing the largest and smallest 
		singular triplets}\label{sec4}
	This section describes the Lanczos bidiagonalization process for tensors using the 
	t-product, and discusses how approximations of the largest and smallest singular triplets 
	of a large third-order tensor $\mathcal{A}\in \K^{\ell\times p}_n$ can be computed.

	\subsection{The tensor Lanczos bidiagonalization algorithm}
	The Lanczos bidiagonalization process was introduced for matrices by Golub and Kahan 
	\cite{golub1965calculating} and therefore sometimes is referred to as the Golub-Kahan
	bidiagonalization process. For a matrix $A\in\R^{\ell\times p}$, this process is closely 
	related to symmetric Lanczos process applied to the real symmetric matrices $AA^T$ and $A^T A$,
	or alternatively to the symmetric matrix
	\[
	\begin{bmatrix}
		0 & A\\
		A^T & 0
	\end{bmatrix}. 
	\]
	
	Lanczos bidiagonalization algorithms have been applied to solve numerous problems such as 
	large-scale least squares problem \cite{paigesaunders}, the approximation of the largest or 
	smallest singular triplets of a large matrix 
	\cite{baglama2005augmented,hochstenbach2001jacobi,kokiopoulou2004computing}, and in 
	Tikhonov regularization of large linear discrete ill-posed problems; see, e.g., \cite{Bj,CR}.
	We note that the bidiagonalization method described in 
	\cite{paigesaunders} and applied in \cite{Bj,CR} reduces a large matrix $A$ to a small lower
	bidiagonal matrix, while in \cite{baglama2005augmented} the matrix $A$ is reduced to a
	small upper bidiagonal matrix. We will review the latter approach.
	
	Application of $m\ll\min\{\ell,p\}$ steps of the Lanczos bidiagonalization process to the
	matrix $A\in\R^{\ell\times p}$ with the initial unit vector $p_1\in\R^\ell$ generically 
	produces two matrices 
	\begin{equation*}
		P_m=\left[p_1,p_2,\ldots,p_m\right]\in\R^{p\times m},\quad 
		Q_m=\left[q_1,q_2,\ldots,q_m\right]\in\R^{\ell\times m}.
	\end{equation*}
	The columns of $P_m$ and $Q_m$ form orthonormal bases for the Krylov subspaces 
	\begin{eqnarray*}
		\mathrsfs{K}_m\left(A^T A, p_1\right)&=&{\tt span}\{p_1,A^TAp_1,\left(A^TA\right)^2p_1,
		\ldots,\left(A^TA\right)^{m-1}p_1 \},\\
		\mathrsfs{K}_m\left(AA^T,q_1\right)&=&{\tt span}\{q_1,AA^Tq_1,\left(AA^T\right)^2q_1,
		\ldots,\left(AA^T\right)^{m-1}q_1\},
	\end{eqnarray*}
	respectively, where $q_1=Ap_1/\|Ap_1\|_2$. A matrix interpretation of the recursion 
	relations of the Lanczos process gives the matrix relations 
	\begin{eqnarray}
		&AP_m=& Q_m B_m, \label{LBm 1}\\
		&A^T Q_m =& P_m B_m^T + \beta_m p_{m+1} e_m^T, \label{LBm 2}
	\end{eqnarray}
	where $e_m=[0,\ldots,0,1]^T\in\R^m$, $\beta_m\geq 0$ is a scalar, and $p_{m+1}\in\R^p$. 
	The matrix $B_m\in\R^{m\times m}$ is upper bidiagonal and satisfies $B_m=Q_m^TAP_m$.
	
	When considering bidiagonalization of a third-order tensor $\mathcal{A}$ using the 
	t-product, the scalars and the columns of the matrices $P_m$ and $Q_m$ in the matrix
	decompositions \eqref{LBm 1} and \eqref{LBm 2} become tubes and lateral slices,
	respectively, in 
	the decompositions determined by the tensor Lanczos bidiagonalization process. The
	application of $m$ steps of tensor Lanczos bidiagonalization to the third-order 
	tensor $\mathcal{A}\in\K^{\ell\times p}_n$ generically computes two tensors 
	\[
	\mathcal{P}_m=\left[\vec{\mathcal{P}}_1,\vec{\mathcal{P}}_2,\dots,
	\vec{\mathcal{P}}_m\right]\in\K^{p\times m}_n \mbox{~~and~~}
	\mathcal{Q}_m=\left[\vec{\mathcal{Q}}_1,\vec{\mathcal{Q}}_2,\dots,
	\vec{\mathcal{Q}}_m\right]\in \K^{\ell\times m}_n,
	\]
	whose lateral slices form bases for the tensor Krylov subspaces 
	$\mathrsfs{K}_m\left(\mathcal{A}^H\star\mathcal{A},\vec{\mathcal{P}}_1\right)$ and 
	\hfill\break 
	$\mathrsfs{K}_m\left(\mathcal{A}\star\mathcal{A}^H,\vec{\mathcal{Q}}_1\right)$, 
	respectively. They are defined by
	\begin{eqnarray*}
		\mathrsfs{K}_m\left(\mathcal{A}^H\star\mathcal{A},\vec{\mathcal{P}}_1\right)&=&
		{\tt span}\{\vec{\mathcal{P}}_1,\left(\mathcal{A}^H \star\mathcal{A}\right)\star 
		\vec{\mathcal{P}}_1,\ldots,\left(\mathcal{A}^H\star\mathcal{A}\right)^{m-1}\star 
		\vec{\mathcal{P}}_1 \}, \\
		\mathrsfs{K}_m\left(\mathcal{A}\star\mathcal{A}^H,\vec{\mathcal{Q}}_1\right)&=&
		{\tt span}\{\vec{\mathcal{Q}}_1,\left(\mathcal{A}\star\mathcal{A}^H\right)\star 
		\vec{\mathcal{Q}}_1,\ldots,\left(\mathcal{A}\star\mathcal{A}^H\right)^{m-1}\star
		\vec{\mathcal{Q}}_1\},
	\end{eqnarray*}
	where $\vec{\mathcal{P}}_1\in\K^p_n$ is a lateral slice of unit norm, and the 
	lateral slice $\vec{\mathcal{Q}}_1\in\K^\ell_n$ is of unit norm and proportional to 
	$\mathcal{A}\star\vec{\mathcal{P}}_1$. Algorithm \ref{alg:6} describes the tensor Lanczos 
	bidiagonalization algorithm.
	
	\begin{algorithm}[H]
		\caption{Tensor Lanczos bidiagonalization using the t-product.}\label{alg:6}
		\textbf{Input:} $\mathcal{A}\in\K^{\ell\times p}_{n}$, number of steps 
		$m\leq\min\{\ell,p\}$, $\vec{\mathcal{P}}_1\in\K^{p}_n$ with unit norm. \\
		\textbf{Output:} $\mathcal{P}_m=[\vec{\mathcal{P}}_1,\vec{\mathcal{P}}_2,\ldots, 
		\vec{\mathcal{P}}_m]\in\K^{p\times m}_{n}$ and $\mathcal{Q}_m=[\vec{\mathcal{Q}}_1, 
		\vec{\mathcal{Q}}_2,\ldots,\vec{\mathcal{Q}}_m]\in\K^{\ell\times m}_{n}$ with orthonormal 
		lateral slices, $\mathcal{B}_m\in\K^{m\times m}_n$ a bidiagonal tensor, and 
		$\vec{\mathcal{R}}_m\in\K^{\ell}_m$.
		\begin{algorithmic}[1]
			\STATE $\mathcal{P}_1=\left[\vec{\mathcal{P}}_1\right]$.
			\STATE  $\vec{\mathcal{Q}}_1= \mathcal{A}\star \vec{\mathcal{P}}_1$.
			\STATE  $[\vec{\mathcal{Q}}_1, \bm{\alpha}_1]={\tt Normalize}(\vec{\mathcal{Q}}_1)$. 
			\STATE  $\mathcal{Q}_1=\left[\vec{\mathcal{Q}}_1\right]$, $\mathcal{B}_m(1,1,:)=\bm{\alpha}_1$.
			\FOR {$i=1\; \text{to}\; m$}
			\STATE $\vec{\mathcal{R}}_i=\mathcal{A}^H \star \vec{\mathcal{Q}}_i - \bm{\alpha}_i \star \vec{\mathcal{P}}_i$.
			\STATE Reorthogonalization $\vec{\mathcal{R}}_i=\vec{\mathcal{R}}_i- \mathcal{P}_i \star (\mathcal{P}_i^H \star \vec{\mathcal{R}}_i)$.
			\IF {$i<m$}
			\STATE $[\vec{\mathcal{P}}_{i+1}, \bm{\beta}_i]={\tt Normalize}(\vec{\mathcal{R}}_i)$.
			\STATE  $\mathcal{P}_{i+1}=\left[\mathcal{P}_{i}, \vec{\mathcal{P}}_{i+1}\right]$, $\mathcal{B}_m(i,i+1,:)=\bm{\beta}_i$.
			\STATE  $\vec{\mathcal{Q}}_{i+1}=\mathcal{A}\star \vec{\mathcal{P}}_{i+1}-\bm{\beta}_i \star \vec{\mathcal{Q}}_i.$
			\STATE  Reorthogonalization $\vec{\mathcal{Q}}_{i+1}=\vec{\mathcal{Q}}_{i+1}- \mathcal{Q}_{i} \star (\mathcal{Q}_{i}^H \star \vec{\mathcal{Q}}_{i+1})$.
			\STATE  $[\vec{\mathcal{Q}}_{i+1}, \bm{\alpha}_{i+1}]={\tt Normalize}(\vec{\mathcal{Q}}_{i+1})$.
			\STATE $\mathcal{Q}_{i+1}=\left[\mathcal{Q}_{i}, \vec{\mathcal{Q}}_{i+1}\right]$, $\mathcal{B}_{m}(i+1,i+1,:)=\bm{\alpha}_{i+1}$.
			\ENDIF
			\ENDFOR
		\end{algorithmic}
	\end{algorithm}
	
	We remark that Algorithm \ref{alg:6} differs from the tensor bidiagonalization algorithms
	described in \cite{kilmer2013third,RU2022} in that the former produces an upper bidiagonal
	tensor $\mathcal{B}_m$, while the latter determine a lower bidiagonal tensor. The use of 
	an upper bidiagonal tensor in the present paper is inspired by the choices in 
	\cite{baglama2005augmented,golub1965calculating}. Algorithm \ref{alg:6} is said to break
	down when one of the tensor slices $\vec{\mathcal{R}}_i$ or $\vec{\mathcal{Q}}_{i+1}$ vanishes.
	We comment below on this situation, but note that breakdown is exceedingly rare. 
	
	\begin{theo}
		Generically, Algorithm \ref{alg:6} determines the decompositions 
		\begin{eqnarray}
			\mathcal{A}\star\mathcal{P}_m&=&\mathcal{Q}_m\star\mathcal{B}_m, \label{eq 18}\\
			\mathcal{A}^H\star\mathcal{Q}_m&=&\mathcal{P}_m\star\mathcal{B}_m^H+ 
			\vec{\mathcal{R}}_m\star \vec{\mathcal{E}}_m^H, \label{eq 19}
		\end{eqnarray}
		with $\mathcal{P}_m\in\K^{p\times m}_n$, $\mathcal{Q}_m\in\K^{\ell\times m}_n$, where 
		$\mathcal{P}_m^H\star\mathcal{P}_m=\mathcal{I}_m$ and 
		$\mathcal{Q}_m^H\star\mathcal{Q}_m=\mathcal{I}_m$. The tensor 
		$\vec{\mathcal{E}}_m\in\K^m_n$ is the canonical lateral slice whose elements are zero 
		except for the first element of the $m$th tube, which equals $1$, and 
		$\vec{\mathcal{R}}_m\in\K^{p}_n$ is determined by steps $4$ and $5$ of Algorithm 
		\ref{alg:6} such that $\mathcal{P}_m^H\star\vec{\mathcal{R}}_m=0$. The tensor 
		$\mathcal{B}_m\in\K^{m\times m}_n$ is upper bidiagonal, each of whose frontal slices is an
		upper bidiagonal matrix. Thus, 
		\[
		\mathcal{B}_m=\begin{bmatrix}
			\bm{\alpha}_1 & \bm{\beta}_1 & \textbf{0} & \ldots & \textbf{0} \\
			\textbf{0} & \bm{\alpha}_2 & \bm{\beta}_2 & \textbf{0} & \vdots \\
			\vdots & \ddots & \ddots & \ddots & \vdots \\
			\textbf{0} & \ldots & \ldots & \bm{\alpha}_{m-1} & \bm{\beta}_{m-1}\\
			\textbf{0} & \ldots& \ldots & \textbf{0} & \bm{\alpha}_m
		\end{bmatrix},
		\]
		where $\bm{\alpha}_i$ and $\bm{\beta}_i$ are tubes in $\K_n$.
	\end{theo}
	
	\begin{proof}
		The relations \eqref{eq 18} and \eqref{eq 19} follow immediately from the recursion 
		relations of Algorithm \ref{alg:6}. The orthonormality of the lateral slices of 
		$\mathcal{P}_m$ and $\mathcal{Q}_m$ can be shown by induction. The proof is closely
		related to the proof of the existence of the relations \eqref{LBm 1} and \eqref{LBm 2}, 
		and the properties of the matrices involved. The latter relations are used in 
		\cite{baglama2005augmented}.
	\end{proof}
	
	The Lanczos bidiagonalization process may suffer from loss of orthogonality of the lateral 
	slices of the tensors ${\mathcal P}_m$ and ${\mathcal Q}_m$. Therefore, 
	reorthogonalization is carried out in Lines $5$ and $9$ in Algorithm \ref{alg:6}. We 
	remark that reorthogonalization makes the algorithm more costly both in terms of storage 
	and arithmetic floating point operations. The extra cost may be acceptable as long as the 
	number of steps $m$ is fairly small; see \cite{baglama2005augmented,simon2000low} for
	discussions in the matrix case. 
	
	Let $\vec{\mathcal{R}}_m$ be the tensor whose lateral slices are defined in Line 5. Then  
	\begin{equation}\label{new 25}
		[\vec{\mathcal{P}}_{m+1}, \bm{\beta}_m]={\tt Normalize}\left(\vec{\mathcal{R}}_m\right).
	\end{equation}
	In the rare event that some $\bm{\beta}_j$, $1\leq j<m$, vanishes, Algorithm \ref{alg:6} 
	breaks down. Then the singular tubes of $\mathcal{B}_j$ are singular tubes of 
	$\mathcal{A}$, and the left and right lateral tensor singular slices are obtained as 
	described below. When no breakdown takes place, we can express equation \eqref{eq 19} as
	\[
	\mathcal{A}^H\star\mathcal{Q}_m=\mathcal{P}_{m+1}\star \mathcal{B}^H_{m,m+1},
	\]
	where $\mathcal{P}_{m+1}$ is obtained from $\mathcal{P}_m$ by appending the lateral slice 
	$\vec{\mathcal{P}}_{m+1}$, defined in \eqref{new 25}, to get $\mathcal{P}_{m+1}=
	\left[\mathcal{P}_m,\vec{\mathcal{P}}_{m+1}\right]\in\K^{p\times(m+1)}_n$, and 
	$\mathcal{B}_{m,m+1}\in\K^{m\times(m+1)}_{n}$ is obtained by appending the lateral slice  
	$\bm{\beta_m}\star \vec{\mathcal{E}}_m$ to $\mathcal{B}_m$, i.e., 
	$\mathcal{B}_{m,m+1}=\left[\mathcal{B}_m, \bm{\beta}_m\star \vec{\mathcal{E}}_m\right]$.
	
	We turn to the connection between the partial Lanczos tridiagonalization of a third-order 
	tensor and the partial Lanczos tridiagonalization process of the tensor 
	$\mathcal{A}^H\star\mathcal{A}$. This connection will be used later. Multiplying 
	\eqref{eq 18} from the left by $\mathcal{A}^H$, we get
	\begin{eqnarray}\label{new 27}
		\mathcal{A}^H\star\mathcal{A}\star\mathcal{P}_m&=&\mathcal{A}^H\star\mathcal{Q}_m\star
		\mathcal{B}_m\nonumber\\
		&=&\mathcal{P}_m\star\mathcal{B}_m^H\star\mathcal{B}_m+\vec{\mathcal{R}}_m\star 
		\vec{\mathcal{E}}_m^H \star \mathcal{B}_m\nonumber \\
		&=&\mathcal{P}_m\star\mathcal{B}_m^H\star\mathcal{B}_m+\vec{\mathcal{R}}_m\star
		\vec{\mathcal{E}}_m^H \star \bm{\alpha}_m.
	\end{eqnarray}
	Let $\mathcal{T}_m$ be the symmetric tridiagonal tensor defined by
	\[
	\mathcal{T}_m=\mathcal{B}_m^H\star\mathcal{B}_m\in\K^{m\times m}_n.
	\]
	Then \eqref{new 27} is a partial tensor Lanczos bidiagonalization of 
	$\mathcal{A}^H\star\mathcal{A}$ with initial lateral slice 
	$\vec{\mathcal{P}}_1=\mathcal{P}_m\star\vec{\mathcal{E}}_1$. The lateral slices of 
	$\mathcal{P}_m$ form an orthonormal basis for the tensor Krylov subspace 
	\[
	\mathrsfs{K}_m\left(\mathcal{A}^H\star\mathcal{A},\vec{\mathcal{P}}_1\right)=
	{\tt span}\{\vec{\mathcal{P}}_1,\mathcal{A}^H\star\mathcal{A}\star\vec{\mathcal{P}}_1,
	\left(\mathcal{A}^H\star\mathcal{A}\right)^2\star\vec{\mathcal{P}}_1,\ldots,
	\left(\mathcal{A}^H\star\mathcal{A}\right)^{m-1}\star\vec{\mathcal{P}}_1\}.
	\]
	
	Similarly, multiplying \eqref{eq 19} from the left by $\mathcal{A}$, we obtain
	\[
	\mathcal{A}\star\mathcal{A}^H\star\mathcal{Q}_m=\mathcal{Q}_m\star\mathcal{B}_m\star
	\mathcal{B}_m^H+\mathcal{A}\star\vec{\mathcal{R}}_m\star\vec{\mathcal{E}}^H_m.
	\]
	It follows that the lateral slices of $\mathcal{Q}_m$ form an orthonormal basis for the
	Krylov subspace 
	\[
	\mathrsfs{K}_m\left(\mathcal{A}\star\mathcal{A}^H,\vec{\mathcal{Q}}_1\right)=
	{\tt span}\{\vec{\mathcal{Q}}_1,\mathcal{A}\star\mathcal{A}^H\star \vec{\mathcal{Q}}_1,
	\left(\mathcal{A}\star\mathcal{A}^H\right)^2\star\vec{\mathcal{Q}}_1,\ldots,
	\left(\mathcal{A}\star\mathcal{A}^H\right)^{m-1}\star\vec{\mathcal{Q}}_1\}.
	\]

	\subsection{Approximating singular tubes and singular lateral slices}
	We describe an approach to approximate the largest or smallest singular triplets 
	(singular tubes and associated left and right lateral singular slices) of a large tensor 
	$\mathcal{A}\in\K^{\ell\times p}_{n}$ using restarted partial tensor Lanczos 
	bidiagonalization. Since the tensor $\mathcal{A}$ is large, computing its $k$ largest or 
	smallest singular triplets by determining the t-SVD of $\mathcal{A}$ is very expensive. 
	The idea is to approximate the extreme singular triplets of the tensor $\mathcal{A}$ by 
	determining the extreme singular triplets the bidiagonal tensor $\mathcal{B}_m$, where $m$
	is small. Let $\{\bm{s}_i,\vec{\mathcal{U}}_i,\vec{\mathcal{V}}_i\}$, $1\leq i\leq m$, 
	denote the singular triplets of $\mathcal{B}_m$. They satisfy
	\[
	\mathcal{B}_m\star\vec{\mathcal{V}}_i = \bm{s}_i \star \vec{\mathcal{U}}_i \mbox{~~and~~} 
	\mathcal{B}^H_m\star\vec{\mathcal{U}}_i=\bm{s}_i\star\vec{\mathcal{V}}_i.
	\]
	The $k\leq m$ largest singular triplets of $\mathcal{A}$ are approximated by the triplets 
	$\{\bm{s}_{i,m}^\mathcal{A},\vec{\mathcal{U}}_{i,m}^\mathcal{A},
	\vec{\mathcal{V}}_{i,m}^\mathcal{A} \}$ defined by  
	\begin{equation}\label{eq 3.8}
		\bm{s}_{i,m}^\mathcal{A}=\bm{s}_i,\quad\vec{\mathcal{U}}_{i,m}^\mathcal{A}=\mathcal{Q}_m\star
		\vec{\mathcal{U}}_i,\quad\vec{\mathcal{V}}_{i,m}^\mathcal{A}=\mathcal{P}_m\star 
		\vec{\mathcal{V}}_i,\quad i=1,2,\ldots,k.
	\end{equation}
	For $i=1,2,\ldots,k$, we have
	\begin{eqnarray*}
		\mathcal{A}\star \vec{\mathcal{V}}_{i,m}^\mathcal{A}&=& \mathcal{A}\star\mathcal{P}_m\star
		\vec{\mathcal{V}}_i\\
		&=&\mathcal{Q}_m\star\mathcal{B}_m \star \vec{\mathcal{V}}_i\\
		&=&\mathcal{Q}_m\star\bm{s}_i\star\vec{\mathcal{U}}_i\\
		&=&\mathcal{Q}_m\star\vec{\mathcal{U}}_i \star \bm{s}_i\\
		&=&\vec{\mathcal{U}}_{i,m}^\mathcal{A}\star\bm{s}_{i,m}^\mathcal{A}.
	\end{eqnarray*}
	Similarly,
	\begin{eqnarray} 
		\mathcal{A}^H\star\vec{\mathcal{U}}_{i,m}^\mathcal{A}=\mathcal{A}^H\star\mathcal{Q}_m 
		\star\vec{\mathcal{U}}_i&=&\left(\mathcal{P}_m\star\mathcal{B}_m+\vec{\mathcal{R}}_m\star
		\vec{\mathcal{E}}_m^H\right)\star\vec{\mathcal{U}}_i\nonumber\\
		\label{eq 24}
		&=&\vec{\mathcal{V}}_{i,m}^\mathcal{A}\star\bm{s}_{i,m}^\mathcal{A}+\vec{\mathcal{R}}_m 
		\star\vec{\mathcal{E}}_m^H\star\vec{\mathcal{U}}_i.
	\end{eqnarray}
	
	To accept $\{\bm{s}_{i,m}^\mathcal{A},\vec{\mathcal{U}}_{i,m}^\mathcal{A},
	\vec{\mathcal{V}}_{i,m}^\mathcal{A}\}$ as an approximate singular triplet of 
	$\mathcal{A}$, the remainder term 
	$\vec{\mathcal{R}}_m\star\vec{\mathcal{E}}_m^H\star\vec{\mathcal{U}}_i$ should be small 
	enough. We can bound the remainder term according to
	\begin{eqnarray*}
		\left\Vert\vec{\mathcal{R}}_m\star \vec{\mathcal{E}}_m^H\star\vec{\mathcal{U}}_i
		\right\Vert_F
		&=&\dfrac{1}{\sqrt{n}}\left\Vert{\tt bdiag}\left( \widehat{\vec{\mathcal{R}}_m}\right)
		{\tt bdiag}\left(\widehat{\left(\vec{\mathcal{E}}^H_m \right)}\right){\tt bdiag}
		\left(\widehat{\vec{{\mathcal{U}}_i}}\right)\right\Vert_F  \\
		&\leq& 
		\dfrac{1}{\sqrt{n}}\left\Vert{\tt bdiag}\left(\widehat{\vec{\mathcal{R}}_m}\right)
		\right\Vert_F\left\Vert{\tt bdiag}\left(\widehat{\left(\vec{\mathcal{E}}^H_m\right)}
		\right){\tt bdiag}\left(\widehat{\vec{{\mathcal{U}}_i}}\right)\right\Vert_F \\
		&=& 
		\left\Vert{\tt bdiag}\left(\vec{\mathcal{R}}_m\right)\right\Vert_F\left\Vert {\tt bdiag}
		\left(\widehat{\left(\vec{\mathcal{E}}^H_m\right)}\right){\tt bdiag}
		\left(\widehat{\vec{{\mathcal{U}}_i}} \right)  \right\Vert_F \\
		&=& 
		\left\Vert\bm{\beta}_m\right\Vert_F\sum_{s=1}^n\left\vert
		\widehat{\left(\vec{\mathcal{E}}^H_m\right)}^{(s)}\widehat{\vec{{\mathcal{U}}_i}}^{(s)} 
		\right\vert.
	\end{eqnarray*}
	Analogously as in \cite{baglama2005augmented}, we require for $1\leq s\leq n$ that
	\begin{equation*}
		\left\vert\widehat{\left(\vec{\mathcal{E}}^H_m\right)}^{(s)}
		\widehat{\vec{{\mathcal{U}}_i}}^{(s)}\right\vert
		\leq\delta'\left\Vert\widehat{\mathcal{A}}^{(s)}\right\Vert=\delta'
		\left(\bm{s}^{\widehat{\mathcal{A}}^{(s)}}_{1,m}\right)=\delta
		\left(\bm{s}^{\widehat{\mathcal{A}}}_{1,m}\right)^{(s)},
	\end{equation*}
	for a user-chosen parameter $\delta'>0$, where 
	$\left(\bm{s}^{\widehat{\mathcal{A}}}_{j,m}\right)^{(s)}$ denotes the $s$th element of the
	$j$th approximate singular tube of $\widehat{\mathcal{A}}$. We obtain from eq. 
	\eqref{new 2.21} that
	\begin{equation*}
		\left\Vert\vec{\mathcal{R}}_m\star\vec{\mathcal{E}}_m^H\star\vec{\mathcal{U}}_
		i\right\Vert_F\leq\delta'\left\Vert\bm{\beta}_m\right\Vert_F\sum_{s=1}^{n}\left(
		\bm{s}^{\widehat{\mathcal{A}}}_1\right)^{(s)}=n\delta'\left\Vert\bm{\beta}_m\right\Vert_F
		\left(\bm{s}^{\mathcal{A}}_1\right)^{(1)}=n\delta''\left(\bm{s}^{\mathcal{A}}_1
		\right)^{(1)},
	\end{equation*}
	where $\delta''=\delta'\left\Vert\bm{\beta}_m\right\Vert_F$. The computed approximate 
	singular triplets $\{\bm{s}_{i,m}^\mathcal{A},\vec{\mathcal{U}}_{i,m}^\mathcal{A},
	\vec{\mathcal{V}}_{i,m}^\mathcal{A}\}$, $i=1,2,\ldots,k$, of $\mathcal{A}$ are accepted as
	singular triplets of $\mathcal{A}$ if
	\begin{equation}\label{new 38}
		\left\Vert\vec{\mathcal{R}}_m\star\vec{\mathcal{E}}_m^H\star\vec{\mathcal{U}}_i\right
		\Vert_F\leq\delta\left(\bm{s}^{\mathcal{A}}_{1,m}\right)^{(1)},\quad i=1,2,\ldots k,
	\end{equation}
	for some user-specified parameter $\delta>0$.
	
	To keep the storage requirement fairly small for large-scale problems, we would like the 
	number of steps $m$ of the tensor Lanczos bidiagonalization process to be small. However, 
	when $m$ is small, it may not be possible to approximate the desired singular triplets 
	sufficiently accurately using the available Krylov subspaces 
	$\mathrsfs{K}_m\left(\mathcal{A}^H\star\mathcal{A},\vec{\mathcal{Q}}_1\right)$ and 
	$\mathrsfs{K}_m\left(\mathcal{A}\star\mathcal{A}^H,\vec{\mathcal{P}}_1\right)$. A remedy
	for this situation is to restart the tensor Lanczos bidiagonalization process. The idea 
	is to repeatedly update the initial lateral slices used for the tensor Lanczos 
	bidiagonalization process, and in this way determine a sequence of increasingly more 
	appropriate Krylov subspaces, until the $k$ desired singular triplets have been found 
	with required accuracy. We remark that restarting techniques have been used for computing a few desired
	singular triplets or eigenvalue-eigenvector pairs of a large matrix, where properties of 
	Ritz vectors, harmonic Ritz vectors, and refined Ritz vectors have been exploited; see, 
	e.g., \cite{baglama2005augmented,hochstenbach2001jacobi,jia2010refined,
		Sorensen1992implicit,stathopoulos1998restarting} for details.

	\subsection{Augmentation by Ritz lateral slices}
	Assume that we would like to approximate the $k$ largest singular triplets of 
	$\mathcal{A}\in\R^{\ell\times p\times n}$. To this end, we carry out $m>k$ steps of 
	tensor Lanczos bidiagonalization as described in the previous subsection. The approximate 
	right singular lateral slice $\vec{\mathcal{V}}_{i,m}^\mathcal{A}$ is a Ritz lateral  slice of $\mathcal{A}^H \star \mathcal{A}$ associated with the Ritz tube 
	$\left(\bm{s}_{i,m}^\mathcal{A}\right)^2=\bm{s}_{i,m}^\mathcal{A}\star
	\bm{s}_{i,m}^\mathcal{A}$ for $i\in \{1,2,\ldots,m\}$, and we have
	\begin{eqnarray*}
		\mathcal{A}^H \star\mathcal{A}\star\vec{\mathcal{V}}_{i,m}^{\mathcal{A}}=
		\mathcal{A}^H\star\vec{\mathcal{U}}_{i,m}^\mathcal{A}\star\bm{s}_{i,m}^\mathcal{A} 
		&=&\left(\vec{\mathcal{V}}_{i,m}^\mathcal{A}\star\bm{s}_{i,m}^\mathcal{A}+ 
		\vec{\mathcal{R}}_m\star\vec{\mathcal{E}}_m^H\star\vec{\mathcal{U}}_i\right)\star 
		\bm{s}_{i,m}^\mathcal{A} \\
		&=&\vec{\mathcal{V}}_{i,m}^\mathcal{A}\star\left(\bm{s}_{i,m}^\mathcal{A}\right)^2+
		\vec{\mathcal{R}}_m\star\vec{\mathcal{E}}_m^H\star\vec{\mathcal{U}}_i\star
		\bm{s}_{i,m}^\mathcal{A}.
	\end{eqnarray*}
	
	In what follows we will show some results that will help us to approximate the largest
	or smallest singular triplets of a third-order tensor. The idea behind these results 
	is to find equations that are analogous to \eqref{eq 18} and \eqref{eq 19}, and such that 
	the reduced tensor will contain the $k$ approximate singular tubes among its first $k$ 
	elements on the diagonal, and the right projection tensor will contain the $k$ right Ritz 
	lateral slices among its first $k$ lateral slices, and the left projection tensor will 
	contain the $k$ left Ritz lateral slices among its first $k$ lateral slices. The following
	theorem will be helpful.
	
	\begin{theo}\label{theo 9} 
		Assume that $m$ steps of Algorithm \ref{alg:6} have been applied to the third-order tensor
		$\mathcal{A}\in\K^{\ell\times p}_n$, and suppose that $\bm{\beta}_m$ in 
		\eqref{eq 19} is nonvanishing. Then for $k<m$, we have 
		\begin{eqnarray}
			\mathcal{A}\star\widetilde{\mathcal{P}}_{k+1}&=&\widetilde{\mathcal{Q}}_{k+1}\star
			\widetilde{\mathcal{B}}_{k+1}, \label{theo eq 1}\\
			\mathcal{A}^H\star\widetilde{\mathcal{Q}}_{k+1}&=&\widetilde{\mathcal{P}}_{k+1}\star
			\widetilde{\mathcal{B}}_{k+1}^H+\widetilde{\bm{\beta}}_{k+1}\star 
			\vec{\widetilde{\mathcal{P}}}_{k+2}\star\vec{\mathcal{E}}_{k+1}^H, \label{theo eq 2}
		\end{eqnarray}
		where $\widetilde{\mathcal{P}}_{k+1}\in\K^{p\times(k+1)}_n$ and 
		$\widetilde{\mathcal{Q}}_{k+1}\in\K^{\ell\times(k+1)}_n$ have orthonormal lateral slices,
		and the first $k$ lateral slices of $\widetilde{\mathcal{P}}_m$ are the first $k$ Ritz 
		lateral slices of $\mathcal{A}$, $\widetilde{\mathcal{B}}_{k+1}\in\K^{(k+1)\times(k+1)}_n$
		is an upper triangular tensor, $\vec{\widetilde{\mathcal{P}}}_{k+2}\in\K^p_n$ is a lateral
		slice that is orthogonal to $\widetilde{\mathcal{P}}_{k+1}$, 
		$\widetilde{\bm{\beta}}_{k+1}\in\K_n$, and $\vec{\mathcal{E}}_{k+1}\in\K^{k+1}_n$ is the 
		canonical element under the t-product.
	\end{theo}
	
	\begin{proof}
		Let the Ritz lateral slices $\vec{\mathcal{V}}_{i,m}^\mathcal{A}$ for $1\leq i\leq k$ 
		be associated with the $k$ Ritz tubes of $\mathcal{A}$. Introduce the tensor
		\begin{equation} \label{new 40}
			\widetilde{\mathcal{P}}_{k+1}=\left[\vec{\mathcal{V}}_{1,m}^\mathcal{A}, 
			\vec{\mathcal{V}}_{2,m}^\mathcal{A},\ldots,\vec{\mathcal{V}}_{k,m}^\mathcal{A}, 
			\vec{\mathcal{P}}_{m+1}\right]\in\K^{p\times (k+1)}_n,
		\end{equation}
		where $\vec{\mathcal{P}}_{m+1}$ is given by \eqref{new 25}. Then, using the fact that 
		$\mathcal{A}\star\vec{\mathcal{V}}_{i,m}^\mathcal{A}=\vec{\mathcal{U}}_{i,m}^\mathcal{A}
		\star\bm{s}_{i,m}^\mathcal{A}$ for $i=1,2,\ldots,k$, we obtain 
		\begin{eqnarray}\label{eq 3.25}
			\mathcal{A}\star\widetilde{\mathcal{P}}_{k+1}
			&=&\left[\mathcal{A}\star\vec{\mathcal{V}}_{1,m}^\mathcal{A},\mathcal{A}\star
			\vec{\mathcal{V}}_{2,m}^\mathcal{A},\ldots,\mathcal{A}\star
			\vec{\mathcal{V}}_{k,m}^\mathcal{A},\mathcal{A}\star\vec{\mathcal{P}}_{m+1}\right]
			\nonumber\\
			&=&\left[\vec{\mathcal{U}}_{1,m}^\mathcal{A}\star\bm{s}_{1,m}^\mathcal{A}, 
			\vec{\mathcal{U}}_{2,m}^\mathcal{A}\star\bm{s}_{2,m}^\mathcal{A},\ldots,
			\vec{\mathcal{U}}_{k,m}^\mathcal{A}\star\bm{s}_{k,m}^\mathcal{A},\mathcal{A}\star
			\vec{\mathcal{P}}_{m+1}\right].
		\end{eqnarray}
		Orthogonalizing the term $\mathcal{A}\star\vec{\mathcal{P}}_{m+1}$ against 
		$\{\vec{\mathcal{U}}_{i,m}^\mathcal{A}\}_{i=1:k}$ gives 
		\begin{equation}\label{eq 3.26}
			\mathcal{A}\star\vec{\mathcal{P}}_{m+1}=\sum_{i=1}^{k}\bm{\rho}_i\star 
			\vec{\mathcal{U}}_{i,m}^\mathcal{A}+\vec{\widetilde{\mathcal{R}}}_k,
		\end{equation} 
		where $\vec{\widetilde{\mathcal{R}}}_k$ is orthogonal to 
		$\{\vec{\mathcal{U}}_{i,m}^\mathcal{A}\}_{i=1:k}$, and the $\bm{\rho}_i$ for 
		$i\in \{1,2,\ldots,k\}$ are given by
		\begin{eqnarray*}
			\bm{\rho}_i=\left(\vec{\mathcal{U}}_{i,m}^\mathcal{A}\right)^H\star 
			\left(\mathcal{A}\star\vec{\mathcal{P}}_{m+1}\right)
			&=&\left(\mathcal{A}^H\star
			\vec{\mathcal{U}}_{i,m}^\mathcal{A}\right)^H \star  \vec{\mathcal{P}}_{m+1}\\
			&=&\left(\vec{\mathcal{V}}_{i,m}^\mathcal{A}\star 
			\bm{s}_{i,m}^\mathcal{A}+\vec{\mathcal{R}}_m\star\vec{\mathcal{E}}_m^H\star 
			\vec{\mathcal{U}}_i\right)^H \star  \vec{\mathcal{P}}_{m+1} \\
			&=&\bm{\beta}_m^H\star \left( \vec{\mathcal{U}}_i^H\star \vec{\mathcal{E}}_m\star 
			\vec{\mathcal{P}}_{m+1}^H\right)\star\vec{\mathcal{P}}_{m+1}\\
			&=&\bm{\beta}_m\star\vec{\mathcal{U}}_i^H\star\vec{\mathcal{E}}_m\\
			&=&\bm{\beta}_m\star\left<\vec{\mathcal{U}}_i,\vec{\mathcal{E}}_m\right>,
		\end{eqnarray*}
		because $\bm{\beta}_m=\bm{\beta}_m^H$.
		
		Let $\vec{\widetilde{\mathcal{R}}}_k=\vec{\widetilde{\mathcal{R}'}}_k\star 
		\widetilde{\bm{\alpha}}_{k+1}$ be a normalization of $\vec{\widetilde{\mathcal{R}}}_k$,
		and introduce the tensors
		\begin{equation}\label{new 44}
			\widetilde{\mathcal{Q}}_{k+1}=\left[\vec{\mathcal{U}}_{1,m}^\mathcal{A}, 
			\vec{\mathcal{U}}_{2,m}^\mathcal{A},\ldots,\vec{\mathcal{U}}_{k,m}^\mathcal{A}, 
			\vec{\widetilde{\mathcal{R}'}}_k\right]\in\K^{\ell\times(k+1)}_n
		\end{equation}
		and 
		\begin{equation}\label{new 45}
			\widetilde{\mathcal{B}}_{k+1}=\begin{bmatrix}
				\bm{s}_{1,m}^\mathcal{A} & \textbf{0} & \ldots & \textbf{0} & \bm{\rho}_1 \\
				\textbf{0}&\bm{s}_{2,m}^\mathcal{A}  & \ldots & \textbf{0} & \bm{\rho}_2 \\
				\vdots & \ddots & \ddots & \ddots & \vdots \\
				\textbf{0} & \ldots & \textbf{0} & \bm{s}_{k,m}^\mathcal{A} & \bm{\rho}_k \\
				\textbf{0} & \ldots & \ldots & \textbf{0} & \widetilde{\bm{\alpha}}_{k+1}
			\end{bmatrix}\in\K^{(k+1)\times(k+1)}_n.
		\end{equation}
		Then, from \eqref{eq 3.25} and \eqref{eq 3.26}, we obtain 
		\begin{eqnarray}\label{new 47}
			\mathcal{A}\star\widetilde{\mathcal{P}}_{k+1}
			&=&\left[\vec{\mathcal{U}}_{1,m}^{\mathcal{A}}\star\bm{s}_{1,m}^\mathcal{A}, 
			\vec{\mathcal{U}}_{2,m}^{\mathcal{A}}\star\bm{s}_{2,m}^\mathcal{A},\ldots, 
			\vec{\mathcal{U}}_{k,m}^{\mathcal{A}}\star\bm{s}_{k,m}^\mathcal{A}, 
			\sum_{i=1}^{k}\bm{\rho}_i\star\vec{\mathcal{U}}_{i,m}^\mathcal{A}+
			\vec{\widetilde{\mathcal{R}}}_k  \right]\nonumber \\
			&=& \widetilde{\mathcal{Q}}_{k+1}\star\widetilde{\mathcal{B}}_{k+1}.
		\end{eqnarray}
		On the other hand, as 
		\[
		\mathcal{A}^H\star\widetilde{\mathcal{Q}}_{k+1}=\left[\mathcal{A}^H\star
		\vec{\mathcal{U}}_{1,m}^\mathcal{A},\mathcal{A}^H\star\vec{\mathcal{U}}_{2,m}^\mathcal{A},
		\ldots,\mathcal{A}^H\star\vec{\mathcal{U}}_{k,m}^\mathcal{A},\mathcal{A}^H\star 
		\overrightarrow{\widetilde{\mathcal{R}'}}_k\right],
		\]
		using \eqref{eq 24}, we get 
		\begin{eqnarray*}
			\mathcal{A}^H\star\vec{\mathcal{U}}_{i,m}^\mathcal{A}
			&=&\vec{\mathcal{V}}_{i,m}^\mathcal{A}\star\bm{s}_{i,m}^\mathcal{A}+\vec{\mathcal{R}}_m 
			\star\vec{\mathcal{E}}_m^H\star\vec{\mathcal{U}}_i \\
			&=&\vec{\mathcal{V}}_{i,m}^\mathcal{A}\star\bm{s}_{i,m}^\mathcal{A}+
			\vec{\mathcal{P}}_{m+1}\star\bm{\beta}_m\star\vec{\mathcal{E}}_m^H\star
			\vec{\mathcal{U}}_i \\
			&=&\vec{\mathcal{V}}_{i,m}^\mathcal{A}\star\bm{s}_{i,m}^\mathcal{A}+
			\vec{\mathcal{P}}_{m+1}\star\bm{\rho}_i^H.
		\end{eqnarray*}
		Since
		\[
		\left<\mathcal{A}^H\star \vec{\widetilde{\mathcal{R}}}'_k ,\vec{\mathcal{V}}_{i,m}^\mathcal{A}\right>=\left(\vec{\widetilde{\mathcal{R}}}'_k\right)^H\star\mathcal{A}\star 
		\vec{\mathcal{V}}_{i,m}^\mathcal{A}=\bm{s}_{i,m}^\mathcal{A}\star
		\left(\vec{\widetilde{\mathcal{R}}}'_k\right)^H\star\vec{\mathcal{U}}_{i,m}^\mathcal{A}=\textbf{0},
		\]
		the tensor $\mathcal{A}^H\star\vec{\widetilde{\mathcal{R}'}}_k$ is orthogonal to 
		$\vec{\mathcal{V}}_{i,m}^\mathcal{A}$. Moreover, in view of that 
		$\vec{\mathcal{V}}_{i,m}^\mathcal{A}$ is orthogonal to $\vec{\mathcal{P}}_{m+1}$, we 
		obtain 
		\begin{equation}\label{new 50}
			\mathcal{A}^H\star\overrightarrow{\widetilde{\mathcal{R}'}}_k=
			\bm{\gamma}\star\vec{\mathcal{P}}_{m+1}+\vec{\mathcal{F}}_{k+1},
		\end{equation}
		where $\vec{\mathcal{F}}_{k+1}$ is orthogonal to $\vec{\mathcal{P}}_{m+1}$ as well as to
		$\vec{\mathcal{V}}_{i,m}^\mathcal{A}$. Due to the orthogonality of 
		$\vec{\widetilde{\mathcal{R}}}_{k}$ (or $\overrightarrow{\widetilde{\mathcal{R}'}}_{k}$) 
		to $\left\{\vec{\mathcal{U}}_{i,m}^\mathcal{A}\right\}_{i=1:k}$, the parameter 
		$\bm{\gamma}$ in \eqref{new 50} is given by
		\begin{eqnarray*}
			\bm{\gamma}=\left<\vec{\mathcal{P}}_{m+1},\mathcal{A}^H\star
			\overrightarrow{\widetilde{\mathcal{R}'}}_k\right>&=&\left<\mathcal{A}\star
			\vec{\mathcal{P}}_{m+1},\overrightarrow{\widetilde{\mathcal{R}'}}_k\right> \\
			&=&\left<\sum_{i=1}^{k}\bm{\rho}_i\star\vec{\mathcal{U}}_{i,m}^\mathcal{A}+ 
			\vec{\widetilde{\mathcal{R}}}_{k},
			\overrightarrow{\widetilde{\mathcal{R}'}}_{k}\right> \\
			&=&\left<\vec{\widetilde{\mathcal{R}}}_k,
			\overrightarrow{\widetilde{\mathcal{R}'}}_k\right>=
			\widetilde{\bm{\alpha}}_{k+1}.
		\end{eqnarray*}
		Consequently,
		\begin{eqnarray}\label{eq 41}
			\mathcal{A}^H\star\widetilde{\mathcal{Q}}_{k+1}
			&=&\left[\vec{\mathcal{V}}_{1,m}^\mathcal{A}\star\bm{s}_{1,m}^\mathcal{A}+ 
			\vec{\mathcal{P}}_{m+1}\star\bm{\rho}_1^H,\ldots,\vec{\mathcal{V}}_{k,m}^\mathcal{A}\star 
			\bm{s}_{k,m}^\mathcal{A}+\vec{\mathcal{P}}_{m+1}\star\bm{\rho}_k^H, 
			\widetilde{\bm{\alpha}}_{k+1}\star\vec{\mathcal{P}}_{m+1}+\vec{\mathcal{F}}_{k+1}\right]
			\nonumber \\
			&=&\widetilde{\mathcal{P}}_{k+1}\star\widetilde{\mathcal{B}}_{k+1}^H+
			\vec{\mathcal{F}}_{k+1}\star\vec{\mathcal{E}}_{k+1}^H\nonumber \\
			&=&\widetilde{\mathcal{P}}_{k+1}\star\widetilde{\mathcal{B}}_{k+1}^H+ 
			\widetilde{\bm{\beta}}_{k+1}\star\vec{\widetilde{\mathcal{P}}}_{k+2}\star
			\vec{\mathcal{E}}_{k+1}^H,
		\end{eqnarray}
		where $\widetilde{\bm{\beta}}_{k+1}$ and $\vec{\widetilde{\mathcal{P}}}_{k+2}$ are 
		determined by the normalization of $\vec{\mathcal{F}}_{k+1}$, i.e., 
		$\vec{\mathcal{F}}_{k+1}=\widetilde{\bm{\beta}}_{k+1}\star 
		\vec{\widetilde{\mathcal{P}}}_{k+2}$, because 
		\[
		\widetilde{\mathcal{B}}_{k+1}^H=\begin{bmatrix}
			\bm{s}_{1,m}^\mathcal{A} & \bm{0} & \ldots & \bm{0} & \bm{0}\\
			\bm{0} & \bm{s}_{2,m}^\mathcal{A} & \bm{0} & \ldots  & \bm{0}\\
			\vdots & \ddots& \ddots & \ddots & \vdots \\
			\bm{0}& \ldots & \bm{0} & \bm{s}_{k,m}^\mathcal{A}& \bm{0}\\
			\bm{\rho}_1^H & \bm{\rho}_2^H & \ldots & \bm{\rho}_k^H & \widetilde{\bm{\alpha}}_{k+1}
		\end{bmatrix}\in \K^{(k+1)\times (k+1)}_n.
		\]
		The orthogonality of $\widetilde{\mathcal{P}}_{k+1}$ and $\widetilde{\mathcal{Q}}_{k+1}$
		now follows from the orthogonality of the sequences 
		$\left\{\vec{\mathcal{V}}_{i,m}^\mathcal{A}\right\}_{i=1:k}$ and 
		$\left\{\vec{\mathcal{U}}_{i,m}^\mathcal{A} \right\}_{i=1:k}$, respectively, given by \eqref{eq 3.8}.
	\end{proof}
	
	In the preceding theorem we assumed $\bm{\beta}_m$ to be nonvanishing. If, instead, 
	$\bm{\beta}_m$ vanishes, then the singular tubes of $\mathcal{B}_m$ are singular tubes of
	$\mathcal{A}$, and the left and right singular lateral slices of $\mathcal{A}$ can be
	determined from those of $\mathcal{B}_m$. Similarly, if $\widetilde{\bm{\beta}}_{k+1}$ in
	\eqref{new 50} vanishes, then the singular tubes of $\widetilde{\mathcal{B}}_{k+1}$ are
	singular tubes of $\mathcal{A}$, and the singular lateral slices of $\mathcal{A}$ can be
	determined from $\widetilde{\mathcal{P}}_{k+1}$ and $\widetilde{\mathcal{Q}}_{k+1}$. 
	
	If the $\widetilde{\bm{\beta}}_{k+1}$ is nonvanishing, then we append new lateral slices to 
	$\widetilde{\mathcal{P}}_{k+1}$ and $\widetilde{\mathcal{Q}}_{k+1}$ repeatedly until 
	iteration $m-k$. This is the subject of the following theorem.
	
	\begin{theo}\label{prop 1}
		Assume that $m$ steps of Algorithm \ref{alg:6} have been applied to $\mathcal{A}$ and that 
		eqs. \eqref{new 47} and \eqref{eq 41} hold. If the $\widetilde{\bm{\beta}}_{k+1}$ are
		nonvanishing for $1\leq k<m$, then we have the following relations 
		\begin{eqnarray*}
			\mathcal{A}\star\widetilde{\mathcal{P}}_m&=&\widetilde{\mathcal{Q}}_m\star
			\widetilde{\mathcal{B}}_m,  \\
			\mathcal{A}^H\star\widetilde{\mathcal{Q}}_m&=&\widetilde{\mathcal{P}}_m\star
			\widetilde{\mathcal{B}}_m^H+\widetilde{\bm{\beta}}_m\star
			\vec{\widetilde{\mathcal{P}}}_{m+1}\star \vec{\mathcal{E}}_m^H,
		\end{eqnarray*} 
		where $\widetilde{\mathcal{P}}_m\in\K^{p\times m}_n$ and 
		$\widetilde{\mathcal{Q}}_m\in\K^{\ell\times m}_n$ have orthonormal lateral slices, 
		$\widetilde{\mathcal{B}}_m\in\K^{m\times m}_n$ is an upper triangular, 
		$\widetilde{\bm{\beta}}_m\in\K_n$, $\vec{\widetilde{\mathcal{P}}}_{m+1}\in\K^p_n$ is
		orthogonal to $\widetilde{\mathcal{P}}_m$, and $\vec{\mathcal{E}}_m\in\K^m_n$ is the
		canonical element under the t-product. The first $k$ lateral slices of 
		$\widetilde{\mathcal{P}}_m$ and $\widetilde{\mathcal{Q}}_{m}$ are the same as those of the
		tensors $\widetilde{\mathcal{P}}_{k+1}$ and $\widetilde{\mathcal{Q}}_{k+1}$, respectively, 
		given in Theorem \ref{theo 9}.
	\end{theo}
	
	\begin{proof}
		Let the tensors $\widetilde{\mathcal{P}}_{k+1}$ and $\widetilde{\mathcal{Q}}_{k+1}$ 
		defined in \eqref{new 47} and \eqref{eq 41}, respectively, be represented by
		\[
		\widetilde{\mathcal{P}}_{k+1}=\left[\vec{\widetilde{\mathcal{P}}}_1,
		\vec{\widetilde{\mathcal{P}}}_2,\ldots,\vec{\widetilde{\mathcal{P}}}_{k+1}\right]\in
		\K^{p\times(k+1)}_n
		\]
		and 
		\[
		\widetilde{\mathcal{Q}}_{k+1}=\left[\vec{\widetilde{\mathcal{Q}}}_1,
		\vec{\widetilde{\mathcal{Q}}}_2,\ldots,\vec{\widetilde{\mathcal{Q}}}_{k+1}\right]\in
		\K^{\ell\times(k+1)}_n,
		\]
		and the tensor $\widetilde{\mathcal{P}}_{k+2}$ be given by 
		\[
		\widetilde{\mathcal{P}}_{k+2}=\left[\widetilde{\mathcal{P}}_{k+1}, 
		\vec{\widetilde{\mathcal{P}}}_{k+2}\right]\in\K^{p\times (k+2)}_n.
		\]
		By normalizing the quantity $\left(\mathcal{I}_\ell-\widetilde{\mathcal{Q}}_{k+1}\star
		\widetilde{\mathcal{Q}}_{k+1}^H\right)\star\mathcal{A}\star
		\vec{\widetilde{\mathcal{P}}}_{k+2}$, we obtain the lateral slice 
		$\vec{\widetilde{\mathcal{Q}}}_{k+2}$ such that 
		$\widetilde{\bm{\alpha}}_{k+2}\star\vec{\widetilde{\mathcal{Q}}}_{k+2}=
		\left(\mathcal{I}_\ell-\widetilde{\mathcal{Q}}_{k+1}\star 
		\widetilde{\mathcal{Q}}_{k+1}^H\right)\star\mathcal{A}\star
		\vec{\widetilde{\mathcal{P}}}_{k+2}$. Application of \eqref{theo eq 2} gives 
		\begin{eqnarray}\label{eq 3.41}
			\widetilde{\bm{\alpha}}_{k+2}\star\vec{\widetilde{\mathcal{Q}}}_{k+2}
			&=&\left(\mathcal{I}_\ell-\widetilde{\mathcal{Q}}_{k+1}\star
			\widetilde{\mathcal{Q}}_{k+1}^H\right)\star\mathcal{A}\star 
			\vec{\widetilde{\mathcal{P}}}_{k+2}\nonumber\\
			&=&\mathcal{A}\star\vec{\widetilde{\mathcal{P}}}_{k+2}-\widetilde{\mathcal{Q}}_{k+1}\star
			\widetilde{\mathcal{Q}}_{k+1}^H\star\mathcal{A}\star\vec{\widetilde{\mathcal{P}}}_{k+2} 
			\nonumber \\
			&=&\mathcal{A}\star\vec{\widetilde{\mathcal{P}}}_{k+2}-\widetilde{\mathcal{Q}}_{k+1}\star
			\left(\widetilde{\mathcal{B}}_{k+1}\star\widetilde{\mathcal{P}}_{k+1}^H+ 
			\widetilde{\bm{\beta}}_{k+1}\star\vec{\mathcal{E}}_{k+1}\star 
			\vec{\widetilde{\mathcal{P}}}_{k+2}^H\right)\star\vec{\widetilde{\mathcal{P}}}_{k+2}
			\nonumber \\
			&=&\mathcal{A}\star\vec{\widetilde{\mathcal{P}}}_{k+2}-\widetilde{\bm{\beta}}_{k+1}\star 
			\vec{\widetilde{\mathcal{Q}}}_{k+1}.
		\end{eqnarray}
		
		Consider the tensors 
		\[
		\widetilde{\mathcal{Q}}_{k+2}=\left[\widetilde{\mathcal{Q}}_{k+1}, 
		\vec{\widetilde{\mathcal{Q}}}_{k+2}\right]\in\K^{\ell\times (k+2)}_n
		\]
		and 
		\[
		\widetilde{\mathcal{B}}_{k+2}=\begin{bmatrix}
			\bm{s}_{1,m}^\mathcal{A} & \textbf{0} & \ldots & \textbf{0} & \bm{\rho}_1 & \textbf{0}\\
			\textbf{0} & \bm{s}_{2,m}^\mathcal{A} & \textbf{0} & \ldots & \bm{\rho}_2 & \textbf{0}\\
			\vdots & \ddots & \ddots & \ddots & \vdots & \vdots \\
			\vdots & \ddots & \ddots & \ddots & \vdots & \vdots \\
			\textbf{0} & \ldots & \textbf{0} & \bm{s}_{k,m}^\mathcal{A} & \bm{\rho}_k & \textbf{0}\\
			\textbf{0} & \ldots & \ldots & \textbf{0} & \widetilde{\bm{\alpha}}_{k+1} & 
			\widetilde{\bm{\beta}}_{k+1} \\
			\textbf{0} & \ldots & \ldots & \ldots & \textbf{0}  & \widetilde{\bm{\alpha}}_{k+2} 
		\end{bmatrix}\in\K^{(k+2)\times(k+2)}_n.
		\]
		Using \eqref{theo eq 1} and \eqref{eq 3.41}, we get 
		\[
		\mathcal{A}\star\widetilde{\mathcal{P}}_{k+2}=\widetilde{\mathcal{Q}}_{k+2}\star 
		\widetilde{\mathcal{B}}_{k+2}.
		\]
		
		To determine the lateral slice $\vec{\widetilde{\mathcal{P}}}_{k+3}$, we normalize 
		$\left(\mathcal{I}-\widetilde{\mathcal{P}}_{k+2}\star
		\widetilde{\mathcal{P}}_{k+2}^H\right)\star\mathcal{A}^H\star
		\vec{\widetilde{\mathcal{Q}}}_{k+2}$ so that 
		\[
		\widetilde{\bm{\beta}}_{k+2}\star \vec{\widetilde{\mathcal{P}}}_{k+3}= 
		\left(\mathcal{I}-\widetilde{\mathcal{P}}_{k+2}\star\widetilde{\mathcal{P}}_{k+2}^H\right)
		\star\mathcal{A}^H\star\vec{\widetilde{\mathcal{Q}}}_{k+2}
		\]
		and
		\begin{equation}\label{eq 3.44}
			\widetilde{\bm{\beta}}_{k+2}\star\vec{\widetilde{\mathcal{P}}}_{k+3}= 
			\mathcal{A}^H\star\vec{\widetilde{\mathcal{Q}}}_{k+2}-\widetilde{\bm{\alpha}}_{k+2}\star
			\vec{\widetilde{\mathcal{P}}}_{k+2}.
		\end{equation}
		It now follows from  \eqref{theo eq 1} and \eqref{eq 3.44} that
		\[
		\mathcal{A}^H \star\widetilde{\mathcal{Q}}_{k+2}=\widetilde{\mathcal{P}}_{k+2}\star 
		\widetilde{\mathcal{B}}_{k+2}^H+\widetilde{\bm{\beta}}_{k+2}\star
		\vec{\widetilde{\mathcal{P}}}_{k+3}\star\vec{\mathcal{E}}_{k+2}^H.
		\]
		
		We can continue this procedure until iteration $m-k$ and then obtain
		\[
		\mathcal{A}\star\widetilde{\mathcal{P}}_m=\widetilde{\mathcal{Q}}_m\star
		\widetilde{\mathcal{B}}_m,\quad \mathcal{A}^H \star \widetilde{\mathcal{Q}}_m= 
		\widetilde{\mathcal{P}}_m\star\widetilde{\mathcal{B}}_m^H+\widetilde{\bm{\beta}}_m\star
		\vec{\widetilde{\mathcal{P}}}_{m+1}\star\vec{\mathcal{E}}_m^H,
		\]
		where $\widetilde{\mathcal{P}}_m$ and $\widetilde{\mathcal{Q}}_m$ have orthonormal lateral
		slices and 
		\[
		\widetilde{\mathcal{B}}_m = \begin{bmatrix}
			\bm{s}_{1,m}^\mathcal{A} & \textbf{0} & \ldots & \bm{\rho}_1 & \textbf{0} & \ldots & 
			\textbf{0}\\
			& \ddots  &  & \vdots &  &  & \\
			&  & \bm{s}_{k,m}^\mathcal{A}  & \bm{\rho}_k &  &  &  \\
			&   &      & \widetilde{\bm{\alpha}}_{k+1} & \widetilde{\bm{\beta}}_{k+1}  &  \\
			&  &   &  &  \ddots & \ddots & \\
			& &   &   &  & \widetilde{\bm{\alpha}}_{m-1} &  \widetilde{\bm{\beta}}_{m-1}\\
			&  &  &   &   &  & \widetilde{\bm{\alpha}}_m
		\end{bmatrix}\in\K^{m\times m}_n.
		\]
		This gives the desired result.
	\end{proof}
	
	
	If we would like to compute the smallest singular triplets of $\mathcal{A}$, then we can 
	use the same theorem, but instead of working with the first right singular lateral slices 
	$\vec{\mathcal{V}}_{i,m}^\mathcal{A}$ for $1\leq i \leq k$, we use the last $k$ right 
	singular lateral slices in \eqref{new 40}. The computations are analogous to those 
	described above.

	\subsection{Augmentation by harmonic Ritz lateral slices}
	When the smallest singular values of a matrix $A$ are clustered, their computation
	by the restarted Lanczos bidiagonalization method as described above may require many 
	iterations. In this situation it may be beneficial to instead compute approximations 
	of the smallest singular values of $A$ by seeking to determine approximations of the 
	largest singular values of the matrix $\left(A^T A\right)^{-1}$ without explicitly 
	computing the matrix $\left(A^T A\right)^{-1}$. This was done for the matrix case by using
	computing harmonic Ritz vectors; see \cite{baglama2005augmented,paige1995approximate}. 
	Harmonic Ritz vectors furnish approximations of 
	eigenvectors of $A^TA$ associated with the corresponding harmonic Ritz values. 
	
	In the case of tensors, harmonic Ritz lateral slices furnish approximations of eigenvectors of
	$\mathcal{A}^H\star\mathcal{A}$ associated with harmonic Ritz tubes of 
	$\mathcal{A}^H\star\mathcal{A}$. The harmonic Ritz tubes $\widecheck{\bm{\theta}}_j$ of
	$\mathcal{A}^H\star\mathcal{A}$ associated with the partial tensor tridiagonalization 
	defined in \eqref{new 27} are the eigentubes of the generalized eigenvalue 
	problem
	\begin{equation}\label{eq 52} 
		\left(\left(\mathcal{B}_m^H\star\mathcal{B}_m\right)^2+\bm{\alpha}_m^2\star\bm{\beta}_m^2
		\star\vec{\mathcal{E}}_m\star\vec{\mathcal{E}}_m^H\right)\star
		\vec{\widecheck{\mathcal{\omega}}}_j=\widecheck{\bm{\theta}}_j\star\mathcal{B}_m^H\star
		\mathcal{B}_m\star\vec{\widecheck{\mathcal{\omega}}}_j,\quad  1\leq j \leq m. 
	\end{equation}
	
	The eigenpair $\{\widecheck{\bm{\theta}}_j,\vec{\widecheck{\mathcal{\omega}}}_j\}$ can be 
	computed without forming the tensor $\mathcal{B}_m^H \star \mathcal{B}_m$. Let 
	\begin{equation}\label{eq 3.47}
		\vec{\mathcal{\omega}}_j= \mathcal{B}_m \star \vec{\widecheck{\mathcal{\omega}}}_j.
	\end{equation}
	Using the relations 
	\begin{equation*}
		\bm{\alpha}_m\star\vec{\mathcal{E}}_m^H =\vec{\mathcal{E}}_m^H\star\mathcal{B}_m
		\text{~~and~~}\bm{\alpha}_m\star \vec{\mathcal{E}}_m=\mathcal{B}_m^H\star
		\vec{\mathcal{E}}_m,
	\end{equation*}
	we can write
	\begin{equation*}
		\bm{\alpha}_{m}^2\star\bm{\beta}_m^2\star\vec{\mathcal{E}}_m\star
		\vec{\mathcal{E}}_m^H=\bm{\beta}_m^2\star\mathcal{B}_m^H\star\vec{\mathcal{E}}_m\star 
		\vec{\mathcal{E}}_m^H\star\mathcal{B}_m.
	\end{equation*}
	Therefore,  using \eqref{eq 3.47}, the relation \eqref{eq 52} can be written as 
	\begin{eqnarray*}
		\mathcal{B}_m^H\star\left(\mathcal{B}_m\star\mathcal{B}_m^H\star\mathcal{B}_m+
		\bm{\beta}_m^2\star\vec{\mathcal{E}}_m\star\vec{\mathcal{E}}_m^H\star
		\mathcal{B}_m\right)\star\mathcal{B}_m^{-1}\star\vec{\mathcal{\omega}}_j=
		\widecheck{\bm{\theta}}_j\star\mathcal{B}_m^H\star\mathcal{B}_m\star\mathcal{B}_m^{-1}
		\star\vec{\mathcal{\omega}}_j.
	\end{eqnarray*}
	It follows that
	\begin{equation}\label{eq 54}
		\left(\mathcal{B}_m\star\mathcal{B}_m^H+\bm{\beta}_m^2\star\vec{\mathcal{E}}_m\star
		\vec{\mathcal{E}}_m^H\right)\star\vec{\mathcal{\omega}}_j=\widecheck{\bm{\theta}}_j\star
		\vec{\mathcal{\omega}}_j
	\end{equation}
	and 
	\[
	\left(\mathcal{B}_m\star\mathcal{B}_m^H+\bm{\beta}_m^2\star\vec{\mathcal{E}}_m\star
	\vec{\mathcal{E}}_m^H\right)=\mathcal{B}_{m,m+1}\star\mathcal{B}_{m,m+1}^H.
	\]
	
	In this subsection, we denote the singular triplets of $\mathcal{B}_{m,m+1}$ by 
	$\{\bm{s}_i',\vec{\mathcal{U}}_i',\vec{\mathcal{V}}_i'\}$ for $1\leq i \leq m$, with
	the first $k$ of them being the smallest singular triplets. Recall that we are
	interested in determining approximations of the smallest singular triplets of 
	$\mathcal{A}$. The $k$ smallest singular triplets of $\mathcal{B}_{m,m+1}$ form the 
	tensors
	\begin{eqnarray*}
		\mathcal{U}_k'&=&\left[\vec{\mathcal{U}}_1',\vec{\mathcal{U}}_2',\ldots, 
		\vec{\mathcal{U}}_k'\right]\in\K^{m\times k}_n,\quad \mathcal{V}_k'=
		\left[\vec{\mathcal{V}}_1',\vec{\mathcal{V}}_2',\ldots,\vec{\mathcal{V}}_k'\right]\in
		\K^{(m+1)\times k}_{n},\\
		\mathcal{S}_k'&=&\left[\bm{s}_1'\star\vec{\mathcal{E}}_1,\bm{s}_2'\star
		\vec{\mathcal{E}}_2,\ldots,\bm{s}_k'\star\vec{\mathcal{E}}_k\right]\in\K^{k\times k}_n,
	\end{eqnarray*}
	where 
	\[
	\mathcal{B}_{m,m+1}\star\mathcal{V}_k'=\mathcal{U}_k'\star\mathcal{S}_k'{\rm ~~and~~ } 
	\mathcal{B}_{m,m+1}^H\star\mathcal{U}_k'=\mathcal{V}_k'\star \mathcal{S}_k'. 
	\]
	We obtain from the above equations that
	\[
	\mathcal{B}_{m,m+1}\star\mathcal{B}_{m,m+1}^H\star\mathcal{U}_k'=\mathcal{U}_k'\star 
	\left(\mathcal{S}_k'\right)^2,
	\]
	where
	\[
	\left(\mathcal{S}_k'\right)^2=\left[\left(\bm{s}_1'\right)^2\star\vec{\mathcal{E}}_1,
	\ldots,\left(\bm{s}_k'\right)^2\star\vec{\mathcal{E}}_k\right].
	\]
	Consequently, the eigenpair $\left\{\left(\bm{s}_i'\right)^2,\mathcal{U}_i'\right\}$ 
	satisfies \eqref{eq 54}, and 
	$\left\{\left(\bm{s}_i'\right)^2,\mathcal{B}_m^{-1}\star\mathcal{U}_i'\right\}$ is an
	eigenpair of \eqref{eq 52}. It follows that the harmonic Ritz lateral slice associated 
	with $\widecheck{\bm{\theta}}_j$ is given by 
	\begin{equation}\label{eq 3.56}
		\vec{\widecheck{\mathcal{V}}}_j=\mathcal{P}_m\star\vec{\widecheck{\mathcal{\omega}}}_j=
		\mathcal{P}_m\star\mathcal{B}_m^{-1}\star\vec{\mathcal{U}}_j'.
	\end{equation}
	
	We turn to the computation of the residual of harmonic Ritz lateral slices. Using eqs. 
	\eqref{new 27} and \eqref{eq 54}, we obtain the relations 
	\begin{eqnarray*}
		\mathcal{A}^H \star \mathcal{A}\star \vec{\widecheck{\mathcal{V}}}_j - \widecheck{\bm{\theta}}_j\star \vec{\widecheck{\mathcal{V}}}_j &=& \mathcal{A}^H \star \mathcal{A} \star \mathcal{P}_m \star \vec{\widecheck{\mathcal{\omega}}}_j - \widecheck{\bm{\theta}}_j \star \mathcal{P}_m \star \vec{\widecheck{\mathcal{\omega}}}_j \nonumber \\
		&=& \left(\mathcal{P}_m \star \mathcal{B}_m^H \star \mathcal{B}_m + \bm{\beta}_m \star \vec{\mathcal{E}}_m^H *\mathcal{B}_m \right)\star \vec{\widecheck{\omega}}_j - \widecheck{\bm{\theta}}_j \star \mathcal{P}_m \star \vec{\widecheck{\omega}}_j\nonumber \\
		&=& \mathcal{P}_m \star \mathcal{B}_m^{-1}\star \left( \mathcal{B}_m \star \mathcal{B}_m^H -\bm{\theta}_j*\mathcal{I}_m \right)\star \vec{\omega}_j + \bm{\beta}_m \star \vec{\mathcal{P}}_{m+1}\star \vec{\mathcal{E}}_m^H \star \vec{\omega}_j \nonumber \\
		&=& -\bm{\beta}_m^2 \star \mathcal{P}_m \star \mathcal{B}_m^{-1}*\vec{\mathcal{E}}_m \star \vec{\mathcal{E}}_m^H \star \vec{\omega}_j + \bm{\beta}_m*\vec{\mathcal{P}}_{m+1}\star \vec{\mathcal{E}}_m^H \star \vec{\omega}_j \nonumber \\ 
		&=& \vec{\mathcal{E}}_m^H \star \vec{\mathcal{\omega}}_j \star \bm{\beta}_m \left(\vec{\mathcal{P}}_{m+1} - \bm{\beta}_m \star \mathcal{P}_m \star \mathcal{B}_m^{-1}\star \vec{\mathcal{E}}_m\right).
	\end{eqnarray*}
	It follows that the residual can be expressed as 
	\begin{equation}\label{eq 3.58}
		\vec{\widecheck{\mathcal{R}}}_m=\vec{\mathcal{P}}_{m+1}-\bm{\beta}_m\star\mathcal{P}_m 
		\star\mathcal{B}_m^{-1}\star\vec{\mathcal{E}}_m.
	\end{equation}
	
	We now proceed analogously as in the previous subsection, i.e., we use the smallest
	harmonic Ritz eigentubes of $\mathcal{B}_{m+1,m}^H\star\mathcal{B}_{m+1,m}$ and associated
	eigenslices to approximate the 
	$k$ smallest singular triplets of $\mathcal{A}$. This yields relations that are analogous 
	to \eqref{eq 18} and \eqref{eq 19}. The following theorem provides the details.
	
	\begin{theo}\label{theo 10}
		Apply $m$ steps of Algorithm \ref{alg:6} to the third-order tensor $\mathcal{A}$ and assume
		that the tensor $\mathcal{B}_m$ in \eqref{eq 18} and \eqref{eq 19} is invertible. Then, 
		for $k=1,\ldots,m-1$, we have the relations
		\begin{eqnarray}
			\mathcal{A}\star \widecheck{\mathcal{P}}_{k+1}&=&\widecheck{\mathcal{Q}}_{k+1}\star
			\widecheck{\mathcal{B}}_{k+1}, \label{theo2 eq1}\\
			\mathcal{A}^H\star\widecheck{\mathcal{Q}}_{k+1}&=&\widecheck{\mathcal{P}}_{k+1}\star
			\widecheck{\mathcal{B}}_{k+1}^H+\widecheck{\bm{\beta}}_{k+1}\star 
			\vec{\widecheck{\mathcal{P}}}_{k+2}\star \vec{\mathcal{E}}^H_{k+1},\label{theo2 eq2}
		\end{eqnarray}
		where $\widecheck{\mathcal{P}}_{k+1}\in\K^{p\times(k+1)}_n$ and 
		$\widecheck{\mathcal{Q}}_{k+1}\in\K^{\ell\times(k+1)}_n$ have orthonormal lateral slices 
		and $\widecheck{\mathcal{B}}_{k+1}\in\K^{(k+1)\times (k+1)}_n$ is an upper triangular 
		tensor, where the $k$ first lateral slices of $\widecheck{\mathcal{P}}_{k+1}$ are 
		a t-linear combination of the $k$ first harmonic Ritz lateral slices of $\mathcal{A}$
		with $\vec{\widecheck{\mathcal{P}}}_{k+2}\in\K^{p}_n$ is orthogonal to 
		$\widecheck{\mathcal{P}}_{k+1}$. Moreover, $\vec{\mathcal{E}}_{k+1}\in\K^m_n$ is the 
		canonical lateral slice under the t-product.
	\end{theo}
	
	\begin{proof} 
		Let $\{\vec{\widecheck{\mathcal{V}}}_i\}_{i=1:k}$ be the first $k$ harmonic Ritz lateral 
		slices of $\mathcal{A}$. Using \eqref{eq 3.56} and \eqref{eq 3.58}, we get 
		\begin{eqnarray*}
			\left[\bm{s}_1'\star\vec{\widecheck{\mathcal{V}}}_1,\bm{s}_2'\star 
			\vec{\widecheck{\mathcal{V}}}_2,\ldots,\bm{s}_k'\star
			\vec{\widecheck{\mathcal{V}}}_k,\vec{\widecheck{\mathcal{R}}}_m\right]
			&=&\left[\mathcal{P}_m,\vec{\mathcal{P}}_{m+1}\right]\star 
			\begin{bmatrix}
				\mathcal{B}_m^{-1}\star\mathcal{U}_k'\star\mathcal{S}_k'&-\bm{\beta}_m \star
				\mathcal{B}_m^{-1}\star\vec{\mathcal{E}}_m\\
				\textbf{0} & \bm{e}      	
			\end{bmatrix} \\
			&=&\mathcal{P}_{m+1}\star\begin{bmatrix}
				\mathcal{B}_m^{-1}\star\mathcal{U}_k'\star\mathcal{S}_k'&-\bm{\beta}_m\star
				\mathcal{B}_m^{-1}\star\vec{\mathcal{E}}_m\\
				\textbf{0} & \bm{e}      	
			\end{bmatrix}.
		\end{eqnarray*}
		
		Define the tensor 
		\begin{equation}\label{new 76}
			\mathcal{J}_{k+1}=\begin{bmatrix}
				\mathcal{B}_m^{-1}\star \mathcal{U}_k'\star \mathcal{S}_k' & -\bm{\beta}_m \star
				\mathcal{B}_m^{-1}\star \vec{\mathcal{E}}_m\\
				\textbf{0} & \bm{e}      	
			\end{bmatrix}.
		\end{equation}
		Using the reduced t-QR factorization of $\mathcal{J}_{k+1}$, we get
		\[
		\mathcal{J}_{k+1}=\mathcal{Q}_{k+1}'\star\mathcal{R}_{k+1}',
		\]
		where $\mathcal{Q}_{k+1}'\in\K^{(m+1)\times (k+1)}_n$ has orthonormal lateral slices and
		$\mathcal{R}_{k+1}'\in\K^{(k+1)\times(k+1)}_n$ is an f-upper triangular tensor. This
		factorization can be computed by a simple modification of Algorithm \ref{alg:3}.
		
		Let
		\begin{equation}\label{new 77}
			\widecheck{\mathcal{P}}_{k+1}=\left[\vec{\widecheck{\mathcal{P}}}_1, 
			\vec{\widecheck{\mathcal{P}}}_2,\ldots,\vec{\widecheck{\mathcal{P}}}_{k+1}\right]=
			\mathcal{P}_{m+1}\star\mathcal{Q}_{k+1}'\in\K^{\ell\times (k+1)}_n. 
		\end{equation}
		Then 
		\begin{eqnarray*}
			\mathcal{A}\star\widecheck{\mathcal{P}}_{k+1}&=&\mathcal{A}\star\mathcal{P}_{m+1}\star 
			\mathcal{Q}_{k+1}' \nonumber \\
			&=&\left[\mathcal{A}\star \mathcal{P}_m,\mathcal{A}\star\vec{\mathcal{P}}_{m+1}\right]
			\star\mathcal{Q}_{k+1}' \\
			&=&\left[\mathcal{A}\star\mathcal{P}_m,\mathcal{A}\star\vec{\mathcal{P}}_{m+1}\right]\star
			\mathcal{J}_{k+1}\star\left(\mathcal{R}_{k+1}'\right)^{-1} \\
			&=&\left[\mathcal{A}\star\mathcal{P}_{m}\star\mathcal{B}_m^{-1}\star\mathcal{U}_k'\star
			\mathcal{S}_k',\mathcal{A}\star\vec{\mathcal{P}}_{m+1}-\mathcal{A}\star\mathcal{P}_m \star 
			\bm{\beta}_m\star\mathcal{B}_m^{-1}\star\vec{\mathcal{E}}_m\right]\star
			\left(\mathcal{R}_{k+1}'\right)^{-1} \\
			&=&\left[\mathcal{Q}_m\star\mathcal{U}_k'\star\mathcal{S}_k',\mathcal{A}\star 
			\vec{\mathcal{P}}_{m+1}-\vec{\mathcal{Q}}_m\star\bm{\beta}_m\right]\star 
			\left(\mathcal{R}_{k+1}'\right)^{-1}.
		\end{eqnarray*}
		Define 
		\begin{equation}\label{new 80}
			\widecheck{\mathcal{Q}}_k=\mathcal{Q}_m\star\mathcal{U}_k'\in\K^{p\times k}_n.
		\end{equation}
		Using the orthogonality of 
		$\mathcal{A}\star\vec{\mathcal{P}}_{m+1}-\bm{\beta}_m\star\vec{\mathcal{Q}}_m$ against the
		lateral slices of $\widecheck{\mathcal{Q}}_k$ gives 
		\begin{equation}\label{eq 3.66}
			\widecheck{\bm{\alpha}}_{k+1}\star\vec{\widecheck{\mathcal{Q}}}_{k+1}=-\bm{\beta}_m\star
			\vec{\mathcal{Q}}_m+\mathcal{A}\star \vec{\mathcal{P}}_{m+1}-\widecheck{\mathcal{Q}}_k
			\star\begin{bmatrix}
				\widecheck{\bm{\gamma}}_1 \\
				\widecheck{\bm{\gamma}}_2 \\
				\vdots \\
				\widecheck{\bm{\gamma}}_k \\
			\end{bmatrix},
		\end{equation}
		where $\left\Vert\vec{\widecheck{\mathcal{Q}}}_{k+1}\right\Vert=1$ and 
		$\widecheck{\bm{\alpha}}_{k+1}$ is the tube obtained from the normalization of the tensor 
		\[
		-\bm{\beta}_m\star\vec{\mathcal{Q}}_m+\mathcal{A}\star\vec{\mathcal{P}}_{m+1}- 
		\widecheck{\mathcal{Q}}_k\star \begin{bmatrix}
			\widecheck{\bm{\gamma}}_1 \\
			\widecheck{\bm{\gamma}}_2 \\
			\vdots \\
			\widecheck{\bm{\gamma}}_k \\
		\end{bmatrix}
		\]
		with 
		\[
		\widecheck{\mathcal{Q}}_k^H\star\left(-\bm{\beta}_m\star\vec{\mathcal{Q}}_m+
		\mathcal{A}\star \vec{\mathcal{P}}_{m+1}\right)=\begin{bmatrix}
			\widecheck{\bm{\gamma}}_1 \\
			\widecheck{\bm{\gamma}}_2 \\
			\vdots \\
			\widecheck{\bm{\gamma}}_k \\
		\end{bmatrix}.
		\]
		It follows from \eqref{new 80} and \eqref{eq 3.66} that
		\begin{eqnarray*}
			\mathcal{A}\star\widecheck{\mathcal{P}}_{k+1}
			&=&\begin{bmatrix}
				\mathcal{Q}_m\star\mathcal{U}_k'\star\mathcal{S}_k',\widecheck{\bm{\alpha}}_{k+1}\star 
				\vec{\widecheck{\mathcal{Q}}}_{k+1}+\widecheck{\mathcal{Q}}_k\star
				\begin{bmatrix}
					\widecheck{\bm{\gamma_1}}\\
					\vdots\\
					\widecheck{\bm{\gamma}}_k
				\end{bmatrix}
			\end{bmatrix}\star 
			\left(\mathcal{R}'_{k+1}\right)^{-1} \\
			&=&\left[\mathcal{Q}_m\star\mathcal{U}_k',\vec{\widecheck{\mathcal{Q}}}_{k+1}\right]\star 
			\begin{bmatrix}
				\bm{s}_1' &   &   &  &  \widecheck{\bm{\gamma}}_1 \\
				&  \ddots &    &   &  \vdots \\
				&    &   &  \bm{s}_k' & \widecheck{\bm{\gamma}}_k \\
				&   &   &   &  \widecheck{\bm{\alpha}}_{k+1}
			\end{bmatrix} \star \left(\mathcal{R}'_{k+1}\right)^{-1}.
		\end{eqnarray*}
		Hence,
		\begin{equation}
			\mathcal{A}\star\widecheck{\mathcal{P}}_{k+1}= \widecheck{\mathcal{Q}}_{k+1}\star 
			\widecheck{\mathcal{B}}_{k+1}, \label{eq 73}
		\end{equation}
		with 
		\begin{equation}\label{new 84}
			\widecheck{\mathcal{B}}_{k+1}=\begin{bmatrix}
				\bm{s}_1' &   &   &  &  \widecheck{\bm{\gamma}}_1 \\
				&  \ddots &    &   &  \vdots \\
				&    &   &  \bm{s}_k' & \widecheck{\bm{\gamma}}_k \\
				&   &   &   &  \widecheck{\bm{\alpha}}_{k+1}
			\end{bmatrix}\star \left(\mathcal{R}_{k+1}'\right)^{-1}\in\K^{(k+1)\times(k+1)}_n,
		\end{equation}
		where $\widecheck{\mathcal{B}}_{k+1}$ is an upper triangular tensor as it is the t-product
		of two upper triangular tensors.
		
		To show \eqref{theo2 eq2}, we first notice that 
		\[
		\mathcal{A}^H\star\widecheck{\mathcal{Q}}_k=\mathcal{A}^H\star\mathcal{Q}_m\star
		\mathcal{U}_k'=\mathcal{P}_{m+1}\star\mathcal{B}_{m,m+1}^H\star\mathcal{U}_k'= 
		\mathcal{P}_{m+1}\star\mathcal{V}_k'\star\mathcal{S}_k'.
		\]
		Using the fact that 
		\[
		\mathcal{B}_{m,m+1}=\left[\mathcal{B}_m,\bm{\beta}_m\star\vec{\mathcal{E}}_m\right]=
		\mathcal{B}_m\star\left[\mathcal{I}_m,\bm{\beta}_m\star\mathcal{B}_m^{-1}\star 
		\vec{\mathcal{E}}_m\right],
		\]
		we get 
		\[
		\mathcal{B}_{m,m+1}\star\mathcal{V}_k'=\mathcal{U}_k'\star\mathcal{S}_k'\Leftrightarrow 
		\left[\mathcal{I}_m,\bm{\beta}_m \star \mathcal{B}_m^{-1}\star\vec{\mathcal{E}}_m\right]
		\star\mathcal{V}_k'=\mathcal{B}_m^{-1}\star\mathcal{U}_k'\star\mathcal{S}_k'.
		\]
		It follows from the above result that
		\[
		\mathcal{V}_k'=\begin{bmatrix}
			\mathcal{B}_m^{-1} \star \mathcal{U}_k'\star\mathcal{S}_k & 
			-\bm{\beta}_m \star \mathcal{B}_m^{-1}\star \vec{\mathcal{E}}_m \\
			\textbf{0} & \bm{e}
		\end{bmatrix}\star 
		\begin{bmatrix}
			\mathcal{I}_k \\
			\vec{\mathcal{E}}_{m+1}^H \star \mathcal{V}_k'
		\end{bmatrix}
		=\mathcal{J}_{k+1}\star 
		\begin{bmatrix}
			\mathcal{I}_k \\
			\vec{\mathcal{E}}_{m+1}^H \star \mathcal{V}_k'
		\end{bmatrix}.
		\]
		We obtain
		\begin{eqnarray*}
			\mathcal{A}^H\star\widecheck{\mathcal{Q}}_k 
			&=&\mathcal{A}^H\star\mathcal{Q}_m\star\mathcal{U}_k' \\
			&=&\mathcal{P}_{m+1}\star\mathcal{B}_{m,m+1}\star\mathcal{U}_k'\\
			&=&\mathcal{P}_{m+1}\star\mathcal{V}_k'\star\mathcal{S}_k' \\
			&=&\mathcal{P}_{m+1}\star\mathcal{J}_{k+1}\star 
			\begin{bmatrix}
				\mathcal{I}_k \\
				\vec{\mathcal{E}}_{m+1}^H \star \mathcal{V}_k'
			\end{bmatrix}\star\mathcal{S}_k' \\
			&=&\mathcal{P}_{m+1}\star\mathcal{Q}_{k+1}'\star\mathcal{R}_{k+1}'\star 
			\begin{bmatrix}
				\mathcal{I}_k \\
				\vec{\mathcal{E}}_{m+1}^H \star \mathcal{V}_k'
			\end{bmatrix} \star \mathcal{S}_k' \\
			&=&\widecheck{\mathcal{P}}_{k+1}\star\mathcal{R}_{k+1}'\star 
			\begin{bmatrix}
				\mathcal{I}_k \\
				\vec{\mathcal{E}}_{m+1}^H \star \mathcal{V}_k'
			\end{bmatrix}\star\mathcal{S}_k'.
		\end{eqnarray*}
		The relation \eqref{eq 73} now yields
		\[
		\widecheck{\mathcal{Q}}_k^H\star\mathcal{A}\star\widecheck{\mathcal{P}}_{k+1}=
		\widecheck{\mathcal{B}}_{k,k+1}\Leftrightarrow\widecheck{\mathcal{P}}_{k+1}^H\star
		\mathcal{A}^H\star\widecheck{\mathcal{Q}}_k=\widecheck{\mathcal{B}}_{k,k+1}^H,
		\]
		where $\widecheck{\mathcal{B}}_{k,k+1}\in\K^{(k+1)\times k}_n$ is the subtensor of 
		$\widecheck{\mathcal{B}}_{k+1}$, which is obtained by removing the last horizontal slice 
		of $\widecheck{\mathcal{B}}_{k+1}$. Then
		\[
		\widecheck{\mathcal{P}}_{k+1}^H\star\mathcal{A}^H\star\widecheck{\mathcal{Q}}_k=
		\mathcal{R}_{k+1}'\star 
		\begin{bmatrix}
			\mathcal{I}_k \\
			\vec{\mathcal{E}}_{m+1}^H \star \mathcal{V}_k'
		\end{bmatrix}\star\mathcal{S}_k'=\widecheck{\mathcal{B}}_{k,k+1}^H
		\]
		and 
		\[
		\widecheck{\mathcal{P}}_{k+1}^H\star\mathcal{A}^H\star\vec{\widecheck{\mathcal{Q}}}_{k+1}=
		\widecheck{\mathcal{B}}_{k+1}^H\star\widecheck{\mathcal{Q}}_{k+1}^H\star
		\vec{\widecheck{\mathcal{Q}}}_{k+1}=\widecheck{\mathcal{B}}_{k+1}^H\star 
		\vec{\mathcal{E}}_{k+1}=\widecheck{\bm{\alpha}}_{k+1}\star\vec{\mathcal{E}}_{k+1}. 
		\]
		Hence,
		\begin{equation}\label{new 93}
			\mathcal{A}^H\star\vec{\widecheck{\mathcal{Q}}}_{k+1}=\widecheck{\bm{\alpha}}_{k+1}\star 
			\vec{\widecheck{\mathcal{P}}}_{k+1}+\vec{\widecheck{\mathcal{R}}}_{k+1}'
		\end{equation}
		with $\vec{\widecheck{\mathcal{R}}}_{k+1}'\perp\widecheck{\mathcal{P}}_{k+1}$. It follows 
		that 
		\[
		\mathcal{A}^H\star\widecheck{\mathcal{Q}}_{k+1}=\widecheck{\mathcal{P}}_{k+1}\star
		\widecheck{\mathcal{B}}_{k+1}^H+\vec{\widecheck{\mathcal{R}}}_{k+1}'\star 
		\vec{\mathcal{E}}_{k+1}^H.
		\]
		Normalization of $\vec{\widecheck{\mathcal{R}}}_{k+1}'$ gives
		\[
		\mathcal{A}^H\star\widecheck{\mathcal{Q}}_{k+1}=\widecheck{\mathcal{P}}_{k+1}\star 
		\widecheck{\mathcal{B}}_{k+1}^H +\widecheck{\bm{\beta}}_{k+1}\star 
		\vec{\widecheck{\mathcal{P}}}_{k+2}\star\vec{\mathcal{E}}_{k+1}^H.
		\]
		The orthonormality of the lateral slices of $\widecheck{\mathcal{P}}_{k+1}$ and 
		$\widecheck{\mathcal{Q}}_{k+1}$ holds by the construction of these tensors. Specifically, 
		it follows from \eqref{new 77} that the lateral slices of $\widecheck{\mathcal{P}}_{k+1}$
		are orthonormal. Due to \eqref{new 80}, the first $k$ lateral slices of 
		$\widecheck{\mathcal{Q}}_{k+1}$ are orthonormal.
	\end{proof}
	
	Notice that if $\widecheck{\bm{\beta}}_{k+1}$ given in \eqref{theo2 eq2} vanishes, then 
	we have determined $k$ singular triplets, i.e., these singular triplets of $\mathcal{A}$ 
	can be computed by using the singular triplets of $\widecheck{\mathcal{B}}_{k+1}$, as well
	as $\widecheck{\mathcal{P}}_{k+1}$ and $\widecheck{\mathcal{Q}}_{k+1}$ defined in 
	\eqref{theo2 eq1} and \eqref{theo2 eq2}. If $\widecheck{\bm{\beta}}_{k+1}$ does not
	vanish, then we append new lateral slices to $\widecheck{\mathcal{P}}_{k+1}$ and 
	$\widecheck{\mathcal{Q}}_{k+1}$ in a similar way as we did in the previous subsection. 
	The following result is analogous to Theorem \ref{prop 1}.
	
	\begin{theo}
		Carry out $m$ steps of Algorithm \ref{alg:6} and assume that eqs \eqref{theo2 eq1} and
		\eqref{theo2 eq2} hold for $k=1,2,\ldots,m-1$. Further, let $\widecheck{\bm{\beta}}_{k+1}$ 
		in \eqref{theo2 eq2} be nonvanishing. Then we have the following relations 
		\begin{eqnarray*}
			\mathcal{A}\star\widecheck{\mathcal{P}}_{m}&=&\widecheck{\mathcal{Q}}_m\star 
			\widecheck{\mathcal{B}}_m,\\
			\mathcal{A}^H\star\widecheck{\mathcal{Q}}_m&=&\widecheck{\mathcal{P}}_m\star 
			\widecheck{\mathcal{B}}_m^H+\widecheck{\bm{\beta}}_m\star 
			\vec{\widecheck{\mathcal{P}}}_{m+1}\star \vec{\mathcal{E}}^H, 
		\end{eqnarray*}
		where $\widecheck{\mathcal{P}}_m\in\K^{p\times m}_n$ and 
		$\widecheck{\mathcal{Q}}_m\in\K^{\ell\times m}_n$ are orthonormal tensors,
		$\widecheck{\mathcal{B}}_m\in\K^{m\times m}_n$ is an upper triangular tensor, 
		$\widecheck{\bm{\beta}}_m$ is a tube of $n$ elements, 
		$\vec{\widecheck{\mathcal{P}}}_{m+1}\in\K^p_n$ is orthogonal to all the lateral slices of
		$\widecheck{\mathcal{P}}_m$ and $\vec{\mathcal{E}}^H\in\K^\ell_n$ is the canonical lateral
		slice under the t-product, where the first $k$ lateral slices of 
		$\widecheck{\mathcal{P}}_m$ and $\widecheck{\mathcal{Q}}_m$ are the same as the lateral 
		slices of $\widecheck{\mathcal{P}}_{k+1}$ and $\widecheck{\mathcal{Q}}_{k+1}$, 
		respectively, given in Theorem \ref{theo 10}.
	\end{theo}
	
	\begin{proof}
		These results can be shown similarly as Theorem \ref{prop 1}. 
	\end{proof}

	Theorem \ref{theo 10} requires the invertibility of $\mathcal{B}_m$. Notice that this 
	tensor is well conditioned if all the frontal slices of $\widehat{\mathcal{B}}_m$ are well
	conditioned, i.e., if
	\[
	\max_{1\leq i\leq n} \kappa\left(\widehat{\mathcal{B}}_m^{(i)}\right)
	\]
	is small, where
	\[
	\kappa(\widehat{\mathcal{B}}_m^{(i)})=
	\dfrac{\left(\widehat{\bm{s}}_1^{\mathcal{B}_m}\right)^{(i)}}
	{\left(\widehat{\bm{s}}_m^{\mathcal{B}_m} \right)^{(i)}}.
	\]
	
	Algorithm \ref{alg:7} describes computations required to compute approximations of either 
	the $k$ largest singular triplets or the $k$ smallest singular triplets of a third-order 
	tensor $\mathcal{A}$  using the methods we developed in the present or previous 
	subsections.
	
	\begin{algorithm}[H]
		\caption{Tensor Lanczos Bidiagonalization Ritz (t-LBR) algorithm  for computing  the largest and the smallest singular triplets.}	\label{alg:7}
		\textbf{Input:} $\mathcal{A}\in\K^{\ell\times p}_{n}$.\\
		\hspace*{1.2 cm }$m$: the number of tensor Lanczos bidiagonalization steps.\\
		\hspace*{1.2 cm }$\vec{\mathcal{P}}_1\in \mathbb{K}^{p}_n$ with unit norm. \\
		\hspace*{1.2 cm }$k$: the number of the desired singular triplets.\\
		\hspace*{1.2 cm }$\delta$: The tolerance to accept the singular triplets approximated.\\
		\hspace*{1.2 cm }$\epsilon$: machine epsilon.\\
		\hspace*{1.2 cm }type: A Boolean variable for the kind  of augmentation which is either 'Ritz' for Ritz\\
		\hspace*{1.2 cm }augmentation or 'Harm' for harmonic Ritz augmentation.\\
		\textbf{Output:} The $k$ desired singular triplets of $\mathcal{A}$, $\{ \sigma_i, \vec{\mathcal{U}}_i, \vec{\mathcal{V}}_i \}_{i=1:k} $.\\
		\begin{algorithmic}[1]
			\STATE Compute the Partial Lanczos bidiagonalization of $\mathcal{A}$ by Algorithm \ref{alg:6}. 
			\STATE  Compute the t-SVD of $\mathcal{B}_m$ using Algorithm \ref{alg:2}.
			\STATE Check the convergence stated in Equation \eqref{new 38}. If all the $k$ desired singular triplets are well approximated, then exist.
			\STATE  Compute the augmented vectors:
			\IF {type='Ritz' or $\bm{k}(\mathcal{B}_m)> \epsilon^{\frac{1}{2}}$}
			\STATE  Compute the tensors $\mathcal{P}:=\widetilde{\mathcal{P}}_{k+1}$, $\mathcal{Q}:=\widetilde{\mathcal{Q}}_{k+1}$, $\mathcal{B}:=\widetilde{\mathcal{B}}_{k+1}$ and the residual $\vec{\mathcal{F}}_k$ from \eqref{new 40}, \eqref{new 44}, \eqref{new 45} and \eqref{new 50}.
			\ENDIF
			\IF  {type='Harm' and $\bm{k}(\mathcal{B}_m)\leq \epsilon^{\frac{1}{2}}$}
			\STATE Compute the t-SVD of $\mathcal{B}_{m,m+1}$.
			\STATE Compute the t-QR factorization of $\mathcal{J}_{k+1}$  in \eqref{new 76}. 
			\STATE Compute the tensors $\mathcal{P}:=\widecheck{\mathcal{P}}_{k+1}$, $\mathcal{Q}:=\widecheck{\mathcal{Q}}_{k+1}$, $\mathcal{B}:=\widecheck{\mathcal{B}}_{k+1}$ and the residual $\vec{\widecheck{\mathcal{R}}}_m$ from \eqref{new 77}, \eqref{new 80}, \eqref{new 84} and \eqref{new 93}.
			\ENDIF
			\STATE Append $m-k$ lateral slices to $\mathcal{P}$ and $\mathcal{Q}$, and $m-k$ horizontal and lateral slices to $\mathcal{B}$ to obtain $\mathcal{P}_m$, $\mathcal{Q}_m$ and $\mathcal{B}_m$, and
			determine a new residual $\vec{\mathcal{R}}_m$.
			\STATE  Go to 2.
		\end{algorithmic}
	\end{algorithm}

	\section{Multidimensional principal component analysis for facial recognition}\label{sec5}
	Principal component analysis (PCA) is used in numerous areas of science and engineering, 
	such as in data denoising, image classification, and facial recognition. Some approaches to
	color image classification 
	involve conversion of color images to grayscale images to reduce the computational burden,
	because color images are represented by tensors, while gray scale images can be 
	represented by matrices; see \cite{classifi1,classifi3}. However, this conversion 
	entails loss of information. A color image in RGB format can be represented by a 
	third-order tensor. This section discusses the application of PCA to third-order
	tensors. 
	
	PCA when applied to gray-scale face recognition computes a set of characteristics 
	(eigenfaces) corresponding to the main components of the initial set of training images. 
	Recognition is done by projecting the training images into the eigenface subspace, in
	which an image of a person is classified by comparing it with other available images in 
	the eigenface subspace. The main advantages of this procedure are its 
	simplicity, speed, and insensitivity to small changes on the faces. 
	
	When applying PCA to third-order tensors using the t-product, tubes, lateral slices, and
	third-order tensors are analogues of scalars, vectors, and matrices in the eigenface 
	technique for classifying grayscale images.  Using this identification, PCA for third-order
	tensors that represent color images is structurally very similar to PCA for matrices that 
	represent grayscale images. The latter is described in \cite{hached2021multidimensional}. 
	
	Let $N$ training color images $I_1,I_2,\ldots,I_N$ of size $\ell\times p\times n$ be 
	available. They are represented by the third-order tensors 
	$\mathcal{I}_1,\mathcal{I}_2,\ldots,\mathcal{I}_N$ in $\R^{\ell\times p\times n}$. The 
	procedure of recognizing color facial images using third-order tensors is as follows:
	\begin{enumerate}
		\item 
		For each image $I_i$ for $i=1,2,\ldots,N$, we determine a lateral slice 
		$\vec{\mathcal{X}}_i\in\R^{\ell p\times 1\times n}$ by vectorizing each frontal slice, 
		i.e., $\vec{\mathcal{X}}_i^{(s)}={\tt vec}(\mathcal{I}_i^{(s)})$ for $s=1,2,\ldots,n$. We
		then construct a tensor, whose frontal slices are given by $\vec{\mathcal{X}}_i$, i.e., 
		\[
		\mathcal{X}=\left[\vec{\mathcal{X}}_1,\vec{\mathcal{X}}_2,\ldots,\vec{\mathcal{X}}_N
		\right]\in\R^{\ell p\times N\times n}.
		\]
		\item 
		Compute the mean of the frontal slices of $\mathcal{X}$, i.e.,
		\[
		\vec{\mathcal{M}}=\sum_{i=1}^{N}\dfrac{\vec{\mathcal{X}}_i}{N},
		\]
		and let
		\[
		\overline{\mathcal{X}}=[\vec{\overline{\mathcal{X}}}_1,\vec{\overline{\mathcal{X}}}_2,
		\ldots,\vec{\overline{\mathcal{X}}}_N],\qquad 
		\vec{\overline{\mathcal{X}}}_i=\vec{\mathcal{X}}_i-\vec{\mathcal{M}}.
		\]
		\item 
		Determine the first $k$ left singular vectors of $\overline{\mathcal{X}}$. We denote them 
		by $\vec{\mathcal{U}}_1,\ldots,\vec{\mathcal{U}}_k$. Construct the projection subspace
		\begin{equation}\label{Uspace}
			\mathbb{U}_k={\tt span}\left\{\vec{\mathcal{U}}_1,\vec{\mathcal{U}}_2,\ldots,
			\vec{\mathcal{A}}_k \right\}
		\end{equation}
		and let
		\[
		\mathcal{U}_k=\left[\vec{\mathcal{U}}_1,\vec{\mathcal{U}}_2,\ldots,\vec{\mathcal{U}}_k
		\right]\in\R^{\ell p\times k\times n}.
		\]
		\item
		Project each face $I_i$ onto the subspace \eqref{Uspace} to obtain
		$\mathcal{U}_k^H\star \vec{\overline{\mathcal{X}}}_i$. A test image $I_0$ also is 
		projected onto the same space to get 
		$\mathcal{U}_k^H\star\left(\vec{\mathcal{X}}_0-\vec{\mathcal{M}}\right)$. Finally,
		determine the closest image to the test image by computing the minimal distance 
		between the projected test image and all the projected  training images.
	\end{enumerate}
	
The main difference between methods that use PCA for facial recognition is the way that 
	the first (dominant) left singular vectors of $\overline{\mathcal{X}}$ are computed. 
	In the present paper, we use  our proposed method  to compute the
	dominant singular triplets that are used in PCA.  The following algorithm summarises 
	the different steps in our approach.
	
	\begin{algorithm}[H]
		\caption{Facial recognition using tensor Lanczos bidiagonalization with Ritz augmentation.}
		\label{alg:9}
		\begin{algorithmic}[1]
			\STATE \textbf{Input:} Training set of images  $\mathcal{X}$ ($N$ images), mean image $\overline{\mathcal{X}}$, test image  $\mathcal{I}_0$ with its associate  lateral slice $\vec{\mathcal{X}}_0={\tt vec}(\mathcal{I}_0)$;  $m$ the number of tensor Lanczos bidiagonalization algorithm;  $k$ the number of the desired left singular slices.
			\STATE \textbf{Output:} Closest image in the database.
			\STATE $[\mathcal{U}_k, \mathcal{S}_k, \mathcal{V}_k]=\text{t-LBR}(\overline{\mathcal{X}}, m,k)$ using Algorithm \ref{alg:7}. 
			\STATE Project $\overline{\mathcal{X}}$ onto $\mathbb{U}_k$ to get  $\mathcal{P}=\mathcal{U}_k^H \star \overline{\mathcal{X}}$.
			\STATE Project the mean of the test image $I_0$ onto $\mathbb{U}_k$, $\vec{\mathcal{P}}_0=\mathcal{U}_k^H \star \left(\vec{\mathcal{X}}_0 - 	\vec{\mathcal{M}}\right)=\mathcal{U}_k^H \vec{\overline{\mathcal{X}}}_0.$
			\STATE Find $i=\underset{i=1, 2, \ldots, N}{\arg \min} \left\Vert \vec{\mathcal{P}}_0- \vec{\mathcal{P}}_i\right\Vert_F$.
		\end{algorithmic}
	\end{algorithm}

\section{Numerical experiments}\label{sec6}
This section illustrates the performance of Algorithm \ref{alg:7} for detecting the largest 
	or smallest singular triplets when applied to synthetic data, tensor compression, and 
	facial recognition. All computations are carried out on a laptop computer with 
	2.3 GHz Intel Core i5 processors and 8 GB of memory using MATLAB 2018a.

	\subsection{Examples with synthetic data}
	We use synthetic data generated by the MATLAB command ${\tt randn}(\ell,p,n)$, which 
	generates a tensor ${\mathcal A}\in\R^{\ell\times p\times n}$, whose entries are normally 
	distributed pseudorandom numbers with mean zero and variance one. 
	
	\subsubsection{Largest singular values}
	Table \ref{tab1} displays the error in the four largest approximate singular tubes 
	computed by augmentation by Ritz lateral slices (referred to as Ritz in the table) and by
	the partial Lanczos bidiagonalization/Golub-Kahan algorithm (referred to as GK in the
	table) as described in \cite{hached2021multidimensional}, but using the t-product. These 
	errors are given by $\left\Vert\mathcal{S}(i,i,:)-\bm{\Sigma}(i,i,:)\right\Vert_F$ for 
	$i=1,2,3,4$ with $m=20$. Table \ref{tab2} shows the number of iterations required when 
	using augmentation by Ritz lateral slices to approximate the four largest singular 
	triplets for tensors of different sizes and the number of Lanczos bidiagonalization 
	steps $m$. 
	
	\begin{table}[H]
		\centering
		\small\addtolength{\tabcolsep}{-2pt}
		\begin{tabular}{l|l|l|l|l|l|l}
			\hline $i$ & Methods  & $100 \times 100 \times 3$ & $500 \times 500 \times 3$ & 
			$1000\times 1000 \times 3$ & $100\times 100 \times 5$ & $500\times 500 \times 5$ \\
			\hline  \multirow{2}{*}{1} & Ritz  & 7.13e-14
			&  1.60e-13 &  2.27e-13 &  2.85e-14  &  1.63e-13 \\
			&  GK & 8.16e-10 & 0.09  &  0.01  &   3.18e-08       &  0.01  \\
			\hline  \multirow{2}{*}{2} & Ritz & 9.29e-14 & 1.98e-13&  1.56e-13   & 5.62e-14
			& 1.48e-13  \\
			& GK & 1.27e-05 &  0.07 & 0.44 &3.12e-04 & 0.15 \\
			\hline  \multirow{2}{*}{3}   & Ritz & 5.01e-14 & 2.70e-13  &  8.93e-14 & 5.41e-14 & 2.66e-13  \\
			& GK & 0.02 &   0.95 &1.78 & 6.05e-04  & 0.51\\
			\hline   \multirow{2}{*}{4} & Ritz & 3.39e-13  & 4.92e-11  &  9.01e-13  & 3.39e-14 &   6.74e-13\\
			& GK &  0.01 & 1.60 &3.37 & 0.08 &  2.03 \\
			\hline
		\end{tabular}
		\captionof{table}{The Frobenius norm $\left\Vert\mathcal{S}(i,i,:)-\bm{\Sigma}(i,i,:)
			\right\Vert_F$, where $\mathcal{S}(i,i,:)$ denotes the singular tubes computed by either 
			augmentation by Ritz lateral slices (Ritz) or by partial Lanczos bidiagonalization also
			known a partial Golub-Kahan bidiagonalization (GK), and 
			$\bm{\Sigma}(i,i,:)$ stands for the singular tubes determined by the t-SVD method with 
			$m=20$ for $i=1,2,3,4$.} \label{tab1}
	\end{table}
	
	\begin{table}[H]
		\centering
		\small\addtolength{\tabcolsep}{-2pt}
		\begin{tabular}{l|l|l|l|l|l|l|l|l|l|l}
			\hline Ritz  & \multicolumn{2}{l|}{$100\times 100 \times 3 $} & \multicolumn{2}{l|}{$500\times 500 \times 3 $} & \multicolumn{2}{l|}{$1000\times 1000 \times 3 $}& \multicolumn{2}{l|}{$100\times 100 \times 5 $}& \multicolumn{2}{l}{$500\times 500 \times 5 $}\\ 
			\cline{2-11} augmentation  & iter & time &  iter & time&  iter & time&  iter & time&  iter & time\\
			\hline $m=10$ & 15& 0.40 & 29  &  2.84 & 41  & 18.20  & 13 & 0.41  & 29 & 4.09  \\
			\hline $m=20$ &  3 & 0.15  &  5 & 2.14   & 7  & 12.88  & 3 & 0.18 &  5 & 2.91
			\\
			\hline
		\end{tabular}
		\captionof{table}{Number of iterations (iter) needed by the Ritz augmentation method to 
			determine the four largest singular tubes for third-order tensors of different sizes with 
			$m=10,\,20$. The columns with header ``time'' shows the CPU time in seconds.} \label{tab2}
	\end{table}
	Table \ref{tab1} shows the Ritz augmentation method to yield much higher accuracy than
	the GK method. Figures \ref{fig1} and \ref{fig2} display the values of some frames of the
	first $10$ singular tubes of third-order tensors of sizes $100 \times 100 \times 3$ and 
	$1000 \times 1000 \times 5$, respectively, computed by Ritz augmentation using Algorithm 
	\ref{alg:7}, the t-SVD, and partial Lanczos bidiagonalization (GK). Each tube is denoted 
	by $\mathcal{S}(k,k,:)\in\K_n$, where $n$ is equal to $3$ or $5$, and $k=1,2,\ldots,10$. 
	In other word, for a fixed $i$ with $1\leq i \leq n$, we plot $\mathcal{S}(k,k,i)\in\K_n$ 
	for $k=1,2,\ldots,10$. As mentioned above, the $i$th computed singular triplet is 
	accepted as an approximate singular triplet if 
	$\vec{\mathcal{R}}_m\star\vec{\mathcal{E}}_m^H\star\vec{\mathcal{U}}_i$ is small enough 
	for $1\leq i\leq k$, where $k$ is the number of desired singular triplets and 
	the $\vec{\mathcal{U}}_i$ are left singular lateral slice of the current tensor 
	$\mathcal{B}_m$; see eq. \eqref{new 38}. Figure \ref{fig3} shows the evolution of the 
	error computed by \eqref{new 38} for the first three singular triplets determined by 
	Algorithm \ref{alg:7} when applied to a third-order tensor of size
	$1000\times 1000\times 3$ for $m=20$.
	
	\begin{figure}[H]
		\centering
		\begin{tabular}{c}
			\includegraphics[width=0.9\linewidth]{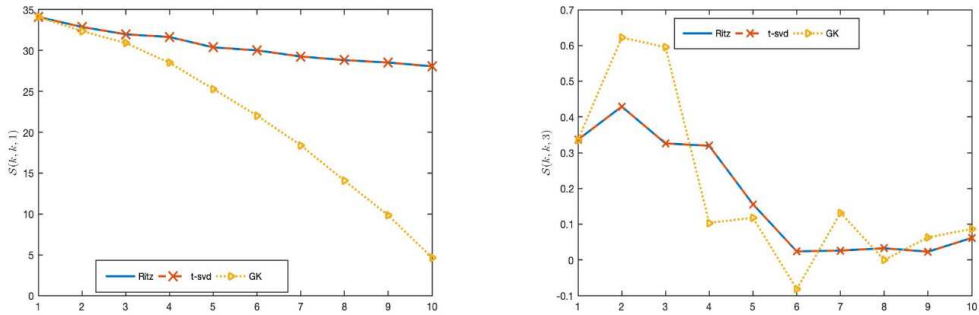}
		\end{tabular}
		\captionof{figure}{On the left, we display the values of the first frontal slices (frames) of the first $10$ singular tubes detected by t-SVD, Ritz augmentation and Partial Lanczos bidiagonalization (GK) for a synthetic data of size $100\times 100 \times 3$  with $m=20$, and on the right we plotted the third frontal slices of these tubes, i.e., $\mathcal{S}(k,k,i)$ with $k=1,2,\ldots,10$ and $i=1,3$.}\label{fig1}
	\end{figure}
	
	Figures \ref{fig1} and \ref{fig2} illustrate that using Algorithm \ref{alg:7} with Ritz 
	augmented method gives more accurate approximations than the GK method. In particular, 
	the frontal slices of each tube computed with Algorithm \ref{alg:7} are very close to the
	corresponding frontal slices of the tubes determined by the t-SVD, independently of the size 
	of the third-order tensor.
	
	\begin{figure}[H]
		\centering
		\begin{tabular}{c}
			\includegraphics[width=0.9\linewidth]{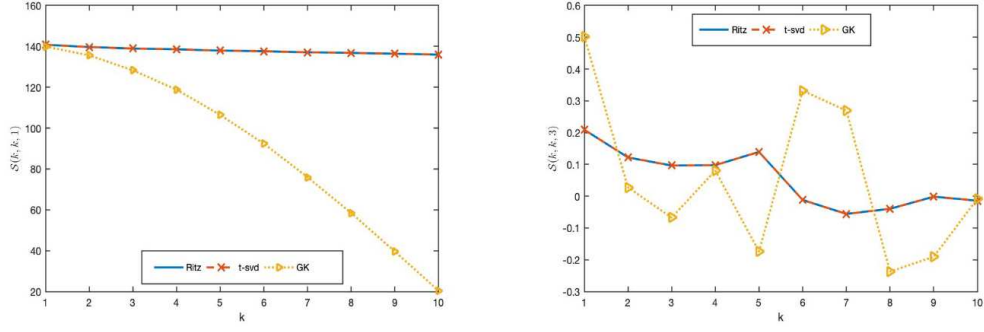}
		\end{tabular}
		\caption{The left-hand side pane shows the values of the first frontal slices (frames) of
			the first $10$ singular tubes computed by t-SVD, Ritz augmentation, and the partial 
			Lanczos bidiagonalization (GK) method for a synthetic data of size 
			$1000\times 1000 \times 5$ with $m=20$. The right-hand side pane displays the third 
			frontal slices of these tubes, i.e., $\mathcal{S}(k,k,i)$ for $k=1,2,\ldots,10$ and 
			$i=1,3$.}\label{fig2}
	\end{figure}
	
	\begin{figure}[H]
		\centering
		\begin{tabular}{c}
			\includegraphics[width=4in, height=2in]{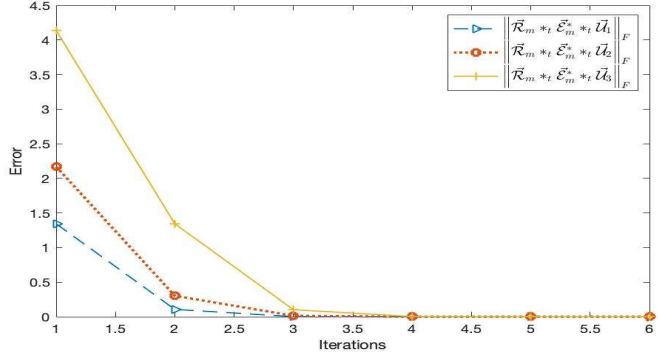}
		\end{tabular}
		\captionof{figure}{Evolution of the remainder term for a third-order tensor of size 
			$1000\times 1000\times 3$ when computing the first three singular triplets by Algorithm
			\ref{alg:7} with Ritz augmentation.}\label{fig3}
	\end{figure}

	\subsubsection{Smallest singular values}
	This subsection illustrates the performance of Algorithm \ref{alg:7} with Ritz 
	augmentation (referred to as Ritz) and with harmonic Ritz augmentation (referred to
	as Harm) for computing the smallest singular triplets of synthetic third-order tensors of 
	different sizes. Table \ref{tab 3} displays the error in the fourth smallest singular tubes 
	computed by Ritz augmentation and harmonic Ritz augmentation for $m=20$, and compares with 
	results determined by the t-SVD method. In Table \ref{tab 4} we show the number of 
	iterations and the required CPU time (in seconds) for these methods when $m=20$.  
	
	\begin{figure}[H]
		\centering
		\begin{tabular}{l|l|l|l|l|l}
			\hline i & Method  & $100\times 100 \times 3$ & $100\times 100 \times 5$  & $500\times 500\times 3$ & $500\times 500 \times 5$\\
			\hline \multirow{2}{*}{$n-3$} & Ritz &  3.82e-11  &  5.22e-12 &    1.34e-10      &  2.50e-10    \\
			& Harm &   1.03e-13   &  4.64e-13   & 4.66e-13    &  1.07e-13 \\
			\hline \multirow{2}{*}{$n-2$} & Ritz &   1.99e-14   & 4.34e-13      &  1.20e-14    &   1.68e-11    \\
			& Harm &   4.94e-15 &  3.10e-13   &   2.46e-14  &  3.77e-14    \\
			\hline \multirow{2}{*}{$n-1$} & Ritz &  8.36e-14  &4.56e-14      &   1.77e-14   &   6.86e-12 \\
			& Harm &   1.64e-15   & 6.05e-15        &  2.88e-14     &   1.39e-13     \\
			\hline \multirow{2}{*}{$n$} & Ritz &  1.38e-15   &  7.71e-16        &     6.49e-15 &   2.00e-12 \\
			& Harm &   8.59e-16   &  7.90e-16         &   3.01e-15   &      1.41e-14 \\
			\hline
		\end{tabular}
		\captionof{table}{The Frobenius norm 
			$\left\Vert\mathcal{S}(i,i,:)-\bm{\Sigma}(i,i,:)\right\Vert_F$, where $\mathcal{S}(i,i,:)$
			denotes the singular tubes determined by Ritz augmentation or harmonic
			Ritz augmentation for $m=20$, and $\bm{\Sigma}(i,i,:)$ are tubes computed by the t-SVD
			method for the four smallest tubes, i.e., for $i=n-3,n-2,n-1,n$.}\label{tab 3}
	\end{figure}
	
	\begin{figure}[H]
		\centering
		\small\addtolength{\tabcolsep}{-1pt}
		\begin{tabular}{l|l|l|l|l|l|l|l|l}
			\hline \multirow{2}{*}{Method} &\multicolumn{2}{l|}{$100\times 100\times 3$}&  \multicolumn{2}{l|}{ $500\times 500\times 3$} & \multicolumn{2}{l|}{ $100\times 100\times 5$} & \multicolumn{2}{l}{ $500\times 500\times 5$}\\
			\cline{2-9} &   CPU time & iter &  CPU time & iter &  CPU time & iter & CPU time & iter \\
			\hline Ritz & 0.99   &  31  & 231.81  & 615 &  1.11 & 30   & 425.83  &   831  \\
			\hline Harm &  0.85  &  29  &  227.49 & 606   &1.03 & 30   &355.35 & 723  \\
			\hline
		\end{tabular}
		\captionof{table}{CPU time in seconds, and number of iterations required by Algorithm 
			\ref{alg:7} with  Ritz augmentation and harmonic Ritz augmentation for $m=20$ to compute 
			the four smallest singular triplets of synthetic third-order tensors of different sizes.}
		\label{tab 4}
	\end{figure}
	
	Tables \ref{tab 3} and \ref{tab 4} show that harmonic Ritz augmentation gives higher 
	accuracy than Ritz augmentation when computing the smallest singular triplets. Figures 
	\ref{fig 4} and \ref{fig 5} depict the Frobenius norm of the remainder term 
	$\vec{\mathcal{R}}_m\star\vec{\mathcal{E}}_m^H\star\vec{\mathcal{U}}_i$ for each 
	iteration with Algorithm \ref{alg:7} with Ritz augmentation and harmonic Ritz augmentation 
	when approximating the last two singular triplets for $m=20$. 
	
	\begin{figure}[H]
		\centering
		\begin{tabular}{c}
			\includegraphics[width=4in, height=2.5in]{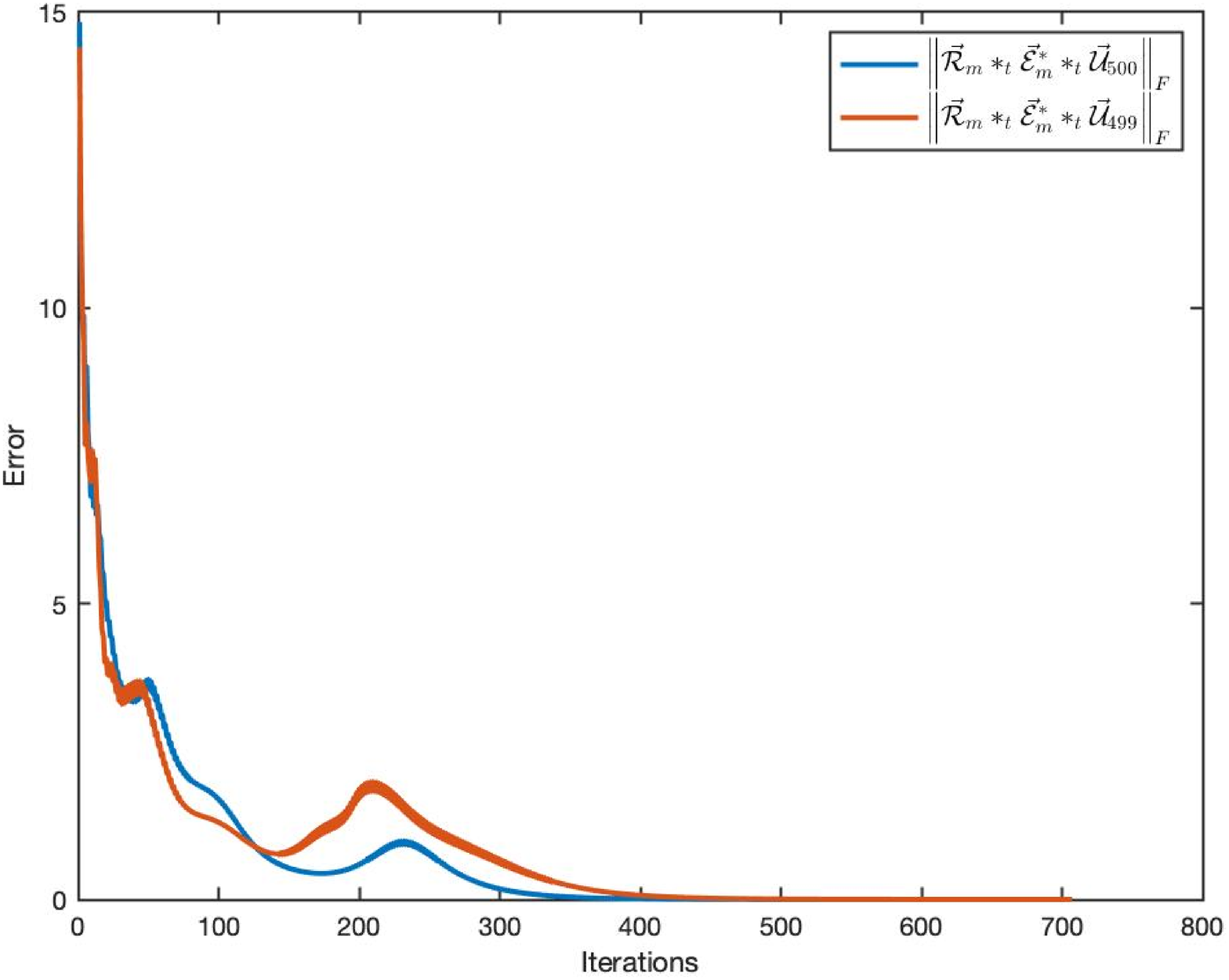}
		\end{tabular}
		\captionof{figure}{The Frobenius norm of 
			$\vec{\mathcal{R}}_m\star \vec{\mathcal{E}}_m^H \star \vec{\mathcal{U}}_i$ obtained by 
			Algorithm \ref{alg:7} with Ritz augmentation when approximating the two smallest
			singular triplets of a synthetic tensor of size $500\times 500\times 5$ with $m=20$ at 
			each iteration for $i=499, 500$.}\label{fig 4}
	\end{figure}
	\begin{figure}[H]
		\centering
		\begin{tabular}{c}
			\includegraphics[width=4in, height=2.5in]{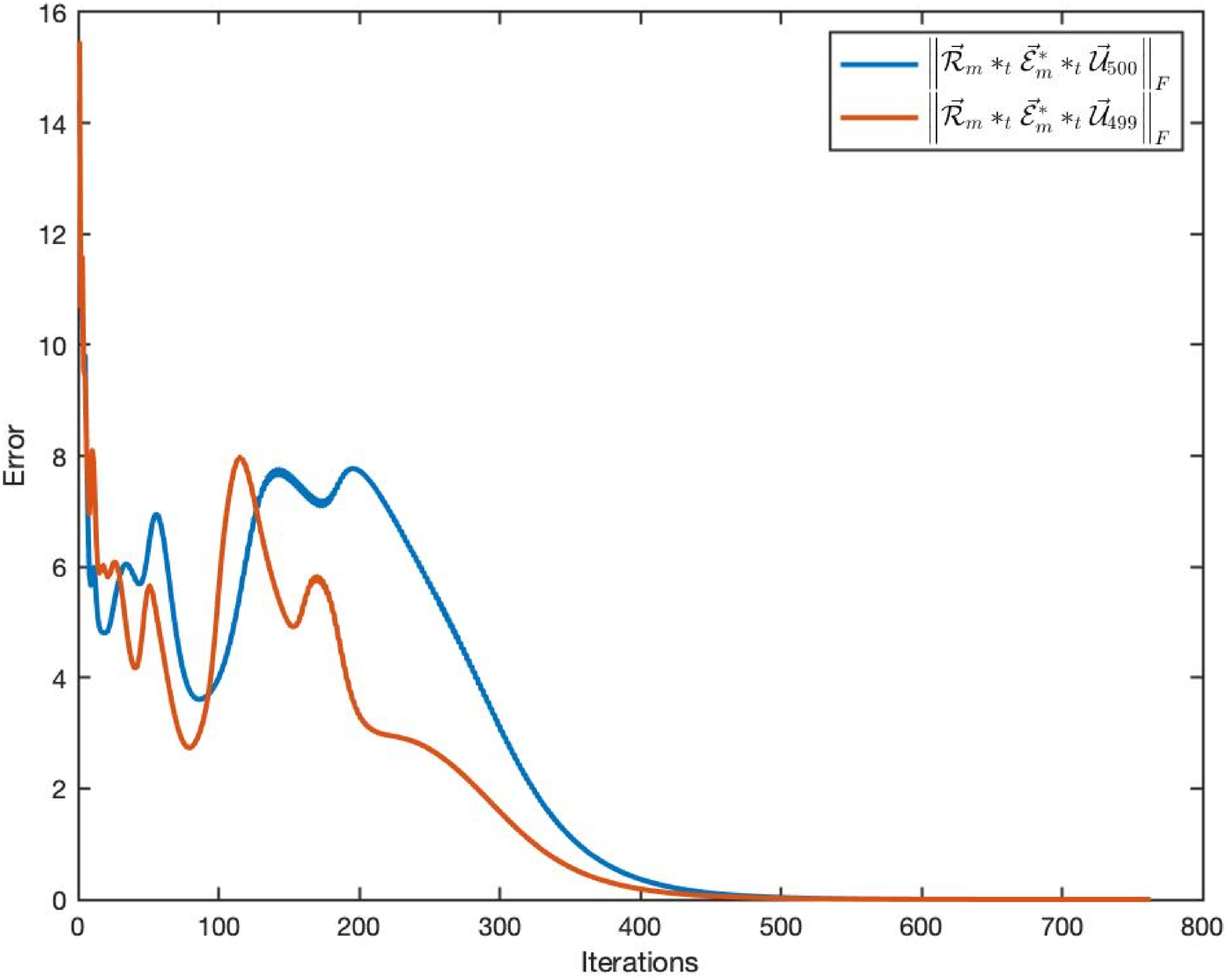}
		\end{tabular}
		\captionof{figure}{The Frobenius norm of $\vec{\mathcal{R}}_m\star \vec{\mathcal{E}}_m^H \star \vec{\mathcal{U}}_i$ obtained by harmonic Ritz augmentation when approximating the last two singular triplets of a synthetic tensor data of size $500\times 500\times 5$ with $m=20$, at each iteration for $i=499, 500$.}\label{fig 5}
	\end{figure}
	Figures \ref{fig 4} and \ref{fig 5} show the error 
	$\|\vec{\mathcal{R}}_m\star\vec{\mathcal{E}}_m^H\star\vec{\mathcal{U}}_i\|_F$ associated
	with Ritz augmentation in Algorithm \ref{alg:7} to converge in a smoother way than the 
	corresponding error for harmonic Ritz augmentation. Both errors converge to zero as the
	number of iterations increases. 
	
	\subsection{Application to data compression}
	Figure \ref{fig 7} displays examples of image compression using two color images: ``house''
	of size $256\times 256\times 3$ and ``Hawaii'' of size $1200\times 1200\times 3$. For each
	image, we compute the $k$th largest singular triplets using Ritz augmentation in 
	Algorithm \ref{alg:7}, which will be referred to as ``Ritz,'' for different numbers $k$ of 
	desired singular triplets. Figure \ref{fig 8} displays the relative error of the 
	compressed images for $k=5,10,15,25$, by using Ritz augmentation (Ritz) and the t-SVD
	method. This error is measured by
	\begin{equation}\label{relerr}
		\dfrac{\left\Vert\mathcal{A}_k-\mathcal{A}\right\Vert_F}
		{\left\Vert\mathcal{A}\right\Vert_F},
	\end{equation}
	where $\mathcal{A}$ denotes the tensor that represents the original image and
	$\mathcal{A}_k=\sum_{i=1}^{k}\vec{\mathcal{U}}_i\star\bm{s}_i\star\vec{\mathcal{V}}^H$.
	
	\begin{figure}[H]
		\centering
		\begin{tabular}{c}
			\includegraphics[width=0.9\linewidth]{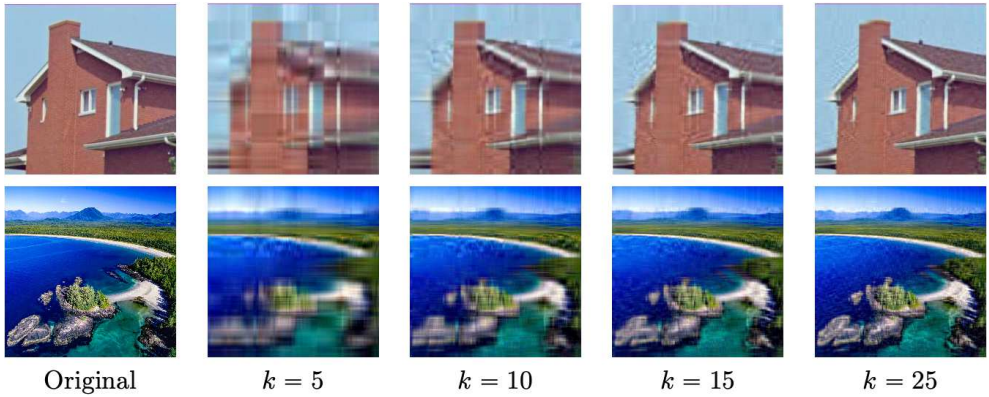}
		\end{tabular}
		\captionof{figure}{Examples of image compression applied to the ``house'' and ``Hawaii''
			images for $k=5,10,15,25$ slices using Algorithm \ref{alg:7} with Ritz augmentation.}
		\label{fig 7}
	\end{figure}
	
	\begin{figure}[H]
		\centering
		\begin{tabular}{cc}
			\includegraphics[width=0.9\linewidth]{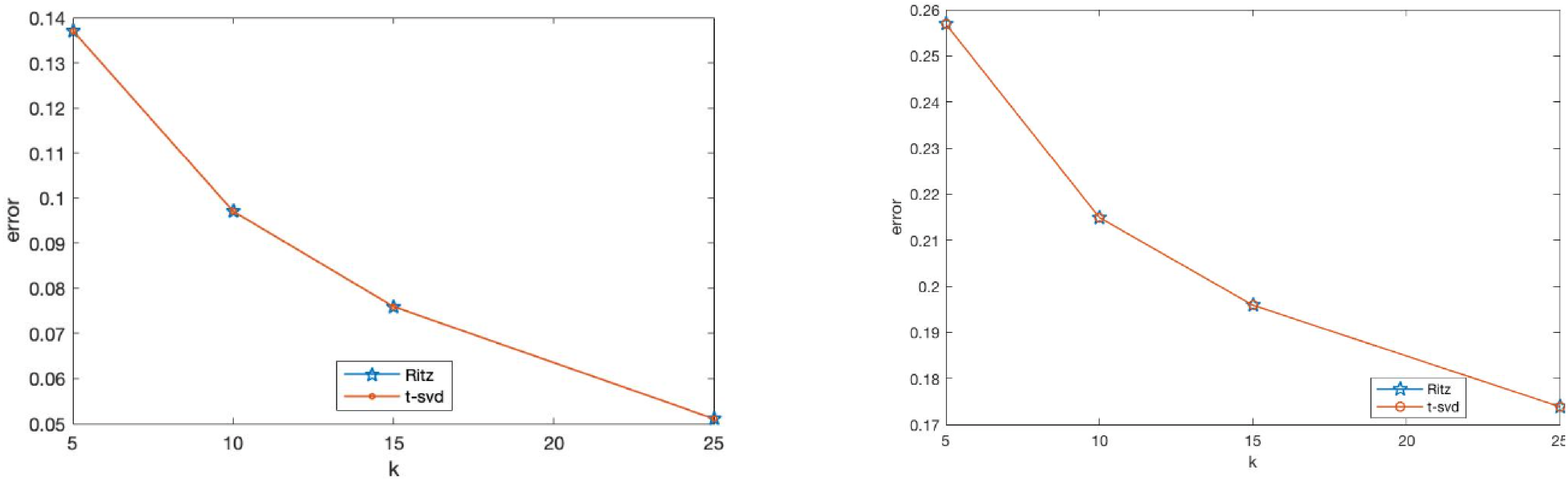}
		\end{tabular}
		\captionof{figure}{Relative compression error \eqref{relerr} for the images ``house'' and
			``Hawaii'' obtained with Algorithm \ref{alg:7} with  Ritz augmentation (Ritz) and the 
			t-SVD method.} \label{fig 8}
	\end{figure}
	
	Figure \ref{fig 8} shows the relative errors obtained with Algorithm \ref{alg:7} with 
	Ritz augmentation and the t-SVD are almost the same. This means that the approximate 
	singular tubes and the right and left singular lateral slices determined by 
	Algorithm \ref{alg:7} with Ritz augmentation are very accurate.

	\subsection{Facial recognition}
	We illustrate the application of Algorithm \ref{alg:9} to facial recognition using color 
	images that are represented by third-order tensors. The images in our test are from the 
	Georgia Tech database GTDB$\_$crop \cite{georg}, which contains $750$ images of $50$ 
	persons, with each person represented by $15$ images that show various facial expressions 
	and facial orientation, and different illumination conditions. Figure \ref{fig 12} shows 
	an example of images of one person in the data set.
	
	\begin{figure}[H]
		\centering
		\begin{tabular}{c}
			\includegraphics[width=0.9\linewidth]{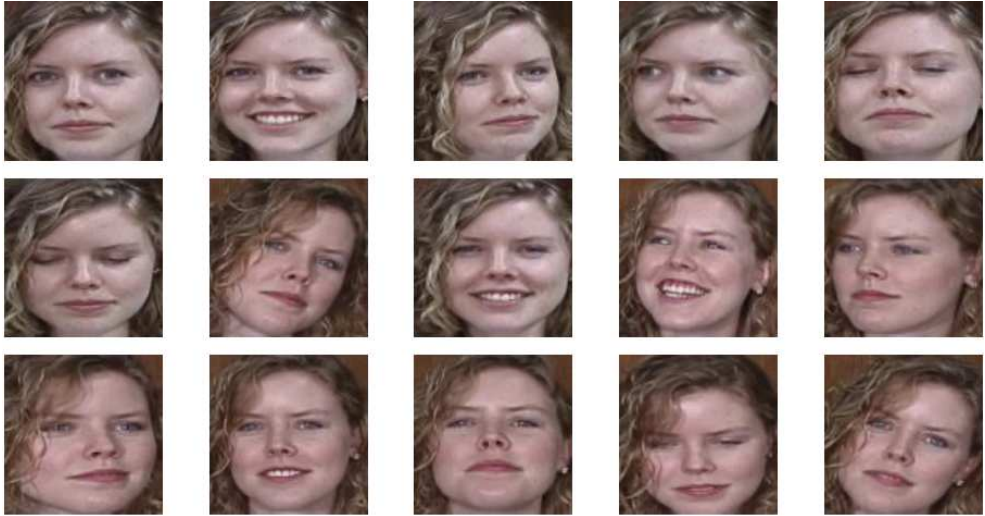}
		\end{tabular}
		\captionof{figure}{An example of a person with different facial expressions and
			orientations.} \label{fig 12}
	\end{figure}
Each image in the data set is of size $100\times 100\times 3$ pixels, and we use $3$ 
	randomly chosen images of each person as test images. The remaining $600$ images form 
	our training set and define the tensor $\mathcal{X}\in\R^{10000\times 600\times 3}$. We 
	applied  Algorithm \ref{alg:9} and compared  the results with  those obtained by 
	the t-SVD and also with results obtained by the` Golub-Kahan (GK) algorithm using
	the t-product. The performance of these methods is measured by the identification rate given by
	\begin{equation}\label{identrate}
		\text{Identification rate}=\dfrac{\text{number of correctly matched images}}
		{\text{number of test images}}\times 100 (\%).
	\end{equation}
Figures \ref{fig 13} and \ref{fig 14} show results obtained for $k=1$ and $k=5$ for two 
	different persons. The mean image is defined as in Algorithm \ref{alg:9}.
	
	\begin{figure}[H]
		\centering
		\begin{tabular}{cc}
			\includegraphics[width=0.5\linewidth]{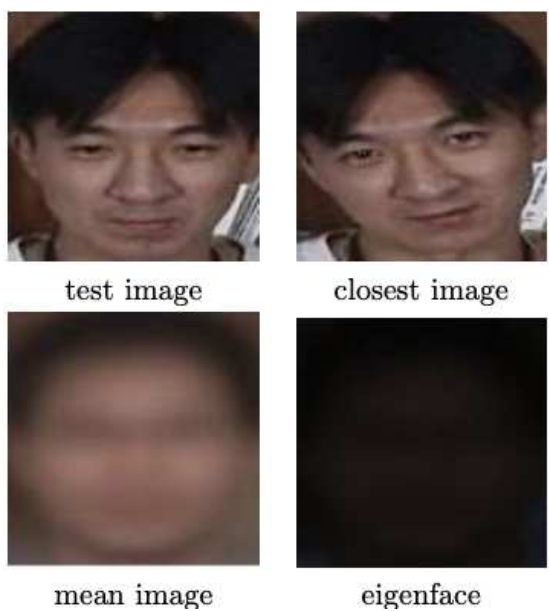}
		\end{tabular}
		\captionof{figure}{A test  for $k=1$.} \label{fig 13}
	\end{figure}
	\begin{figure}[H]
		\centering
		\begin{tabular}{cc}
			\includegraphics[width=0.5\linewidth]{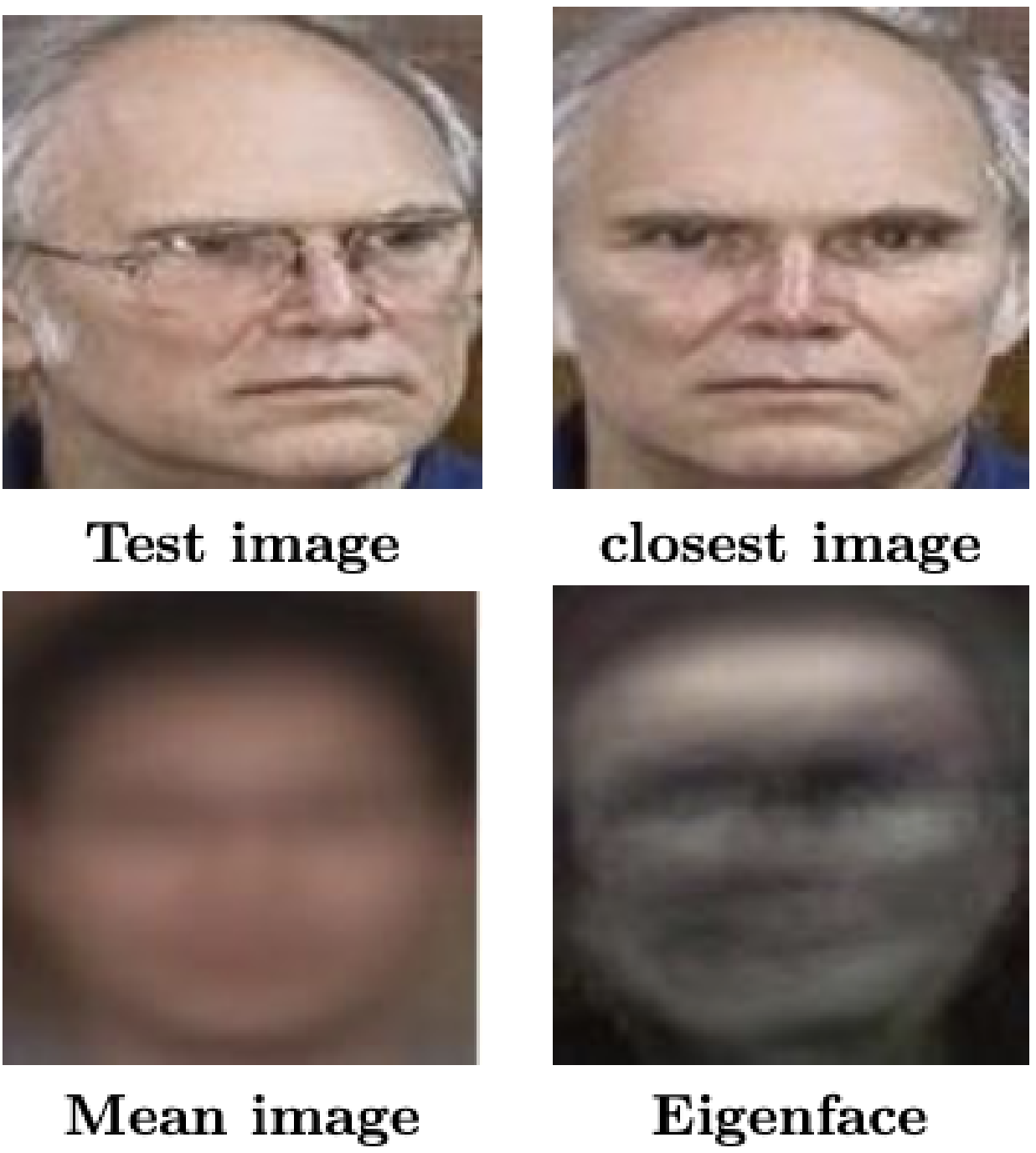}
		\end{tabular}
		\captionof{figure}{A test  for $k=5$.} \label{fig 14}
	\end{figure}
	
	\begin{figure}[H]
		\centering
		\begin{tabular}{c}
			\includegraphics[width=0.7\linewidth]{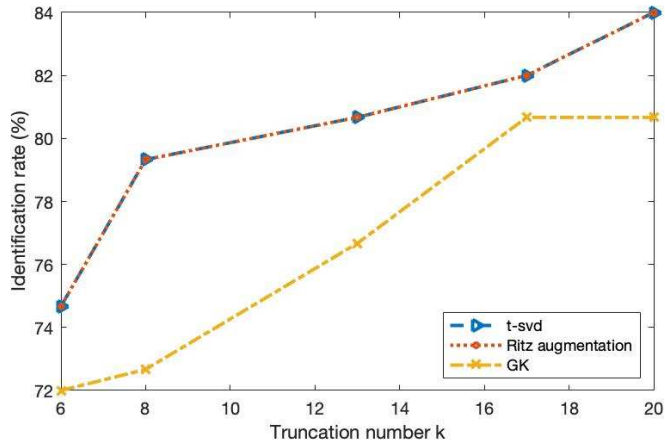}
		\end{tabular}
		\captionof{figure}{Identification rates for different truncation indices $k$ by Ritz augmentation,  t-SVD and Golub-kahan methods.} \label{fig 15}
	\end{figure}
\noindent Figures \ref{fig 13} and \ref{fig 14} show that  Algorithm \ref{alg:9}  performs well for some 
	values of  the truncation index $k$. In Figure \ref{fig 15},  we plotted  the identification rate 
	\eqref{identrate}  obtained with Algorithm \ref{alg:9} (Ritz augmentation), GK for $m=k$, and with 
	the exact t-SVD method   for the  $150$ test images.\\
\begin{table}[H]
	\centering
	\begin{tabular}{l|l|l|l|l|l|l}
		k &  \multicolumn{2}{l}{$2$} &  \multicolumn{2}{|l}{$3$} &  \multicolumn{2}{|l}{$4$} \\
		\hline Method & Ritz & t-SVD & Ritz & t-SVD & Ritz & t-SVD \\
		\hline CPU time (s) &   \textbf{10.60} & 52.82 & \textbf{13.11}  &  63.63 & \textbf{13.88} & 64.77
	\end{tabular}
\captionof{table}{CPU time (in seconds) for Algorithm \ref{alg:9} (Ritz) and for the
t-SVD method for $m=10$ and different values of the truncation index $k$.}
\label{tab 5}
\end{table}

Table \ref{tab 5} reports CPU times for Algorithm \ref{alg:9} for $m=10$ (Ritz) and 
for the t-SVD method for different values of the truncation index $k$. The results show 
Algorithm \ref{alg:9} to be very effective both in terms of accuracy and CPU time 
compared to the t-SVD and the classical Golub-Kahan methods.

	\section{Conclusion and extensions}\label{sec7}
	This paper presents two new methods for approximating the largest or smallest singular
	triplets of large third-order tensors using the t-product. We use restarted Lanczos 
	bidiagonalization for third-order tensors to develop the Ritz augmentation method to 
	determine the largest or smallest singular triplets. Moreover, we propose the 
	harmonic Ritz augmentation method to compute the smallest singular triplets. These methods
	are applied to data compression and face recognition.


\begin{thebibliography}{99}
		\bibitem{arnold2018}
		T. Arnold, M. Kane, B. W. Lewis, A Computational Approach to Statistical Learning, CRC 
		Press, Boca Raton, 2018.
		
		\bibitem{classifi1}
		M. N. Asif, I. S. Bajwa, S. I. Hyder, M. Naweed, Feature based image classification by 
		using principal component analysis, ICGST International Journal on Graphics, Vision and 
		Image Processing. 9, 22--17 (2009).
		
		\bibitem{kilmer-compression}
		H. Avron, L. Horesh, M. E. Kilmer, E. Newman, Tensor-tensor algebra for optimal 
		representation and compression of multiway data, Proceedings of the National Academy of
		Sciences, 188, 28 (2021)
		
		\bibitem{BPP22}
		J. Baglama, V. Perovic, J. Picucci, Hybrid iterative refined restarted Lanczos 
		bidiagonalization method, Numerical Algorithms, in press (2022).
		
		\bibitem{baglama2005augmented}
		J. Baglama, L. Reichel, Augmented implicitly restarted Lanczos bidiagonalization methods, 
		SIAM Journal on Scientific Computing, 27, 19--42 (2005).
		
		\bibitem{lewis2021}
		J. Baglama, L. Reichel, B. W. Lewis, irbla: Fast truncated singular value 
		decomposition and principal component analysis,
		\url{https://cran.r-project.org/web/packages/irlba/index.html}.
		
		\bibitem{BRR13}
		J. Baglama, L. Reichel, D. Richmond, An Augmented LSQR Method, Numerical Algorithms, 
		64, 263--293 (2013).
		
		\bibitem{BR14}
		J. Baglama, D. Richmond, Implicitly restarting the LSQR algorithm, Electronic
		Transactions on Numerical Analysis, 42, 85--105 (2014).
		
		\bibitem{BNR}
		F. P. A. Beik, M. Najafi-Kalyani, L. Reichel, Iterative Tikhonov regularization of tensor 
		equations based on the Arnoldi process and some of its generalizations, Applied 
		Numerical. Mathematics, 151, 425--447 (2020).
		
		\bibitem{elha1}
		A. H. Bentbib, A. EL Hachimi, K. Jbilou, A. Ratnani, A tensor regularized nuclear norm 
		method for image and video completion, Journal of Optimization Theory and Applications, 
		192, 401--425 (2022).
		
		
		\bibitem{Bj}
		\AA. Bj{\"o}rck, A bidiagonalization algorithm for solving large and sparse ill-posed 
		systems of linear equations, BIT Numerical Mathematics, 18, 659--670 (1988).
		
		
		
		\bibitem{CR}
		D. Calvetti, L. Reichel, Tikhonov regularization of large linear problems, BIT Numerical
		Mathematics, 43, 263--283 (2003).
		
		\bibitem{lu2019tensor}
		Y. Chen, J. Feng, H. Lin, W. Liu, C. Lu, S. Yan, Tensor robust principal component 
		analysis with a new tensor nuclear norm, IEEE Transactions on Pattern Analysis and Machine 
		Intelligence, 42, 925--938  (2019).
		
		
		
		\bibitem{golub1965calculating}
		G. Golub, W. Kahan, Calculating the singular values and pseudo-inverse of a matrix, 
		Journal of the Society for Industrial and Applied Mathematics, Series B: Numerical 
		Analysis, 2, 205--224 (1965).
		
		\bibitem{Elguide} M. El Guide, A. El Ichi, K. Jbilou and R. Sadaka, On tensor GMRES and Golub-Kahan methods via the T-product for color image processing, The Electronic Journal of Linear Algebra 37, 524--543 (2021).
		
		\bibitem{ElIchi}
		A. El Ichi, K. Jbilou, R. Sadaka, On tensor tubal-Krylov subspace methods, Linear and 
		Multilinear Algebra, Linear and Multilinear Algebra, 37, 524--543 (2021).
		
		
		
		\bibitem{hached2021multidimensional}
		M. Hached, K. Jbilou, C. Koukouvinos, M. Mitrouli, A multidimensional principal component 
		analysis via the C-product Golub-Kahan-SVD for classification and face recognition, 
		Mathematics, 9, 1249 (2021).
		
		\bibitem{hao2013facial}
		N. Hao, M. E. Kilmer, K. Braman, R. C. Hoover, Facial recognition using tensor-tensor 
		decompositions, SIAM Journal on Imaging Sciences, 6, 437--463 (2013).
		
		\bibitem{hochstenbach2001jacobi}
		M. E. Hochstenbach,  A Jacobi-Davidson type SVD method, SIAM Journal on Scientific 
		Computing, 23, 606--628 (2001).
		
		\bibitem{jia2010refined}
		Z. Jia, D. Niu, A refined harmonic Lanczos bidiagonalization method and an implicitly 
		restarted algorithm for computing the smallest singular triplets of large matrices, SIAM 
		Journal on Scientific Computing, 32, 714--744 (2010).
		
		
		\bibitem{kernfeld2015tensor}
		E. Kernfeld, M. Kilmer, S. Aeron, Tensor-tensor products with invertible linear 
		transforms, Linear Algebra and Its Applications, 485, 545--570  (2015).
		
		\bibitem{kilmer2013third}
		M. E. Kilmer, K. Braman, N. Hao, R. C. Hoover, Third-order tensors as operators on 
		matrices: A theoretical and computational framework with applications in imaging, SIAM 
		Journal on Matrix Analysis and Applications, 34, 148--172 (2013).
		
		\bibitem{kilmer}
		M. Kilmer, C. D. Martin, Factorization strategies for third-order tensors, Linear Algebra 
		and Its Applications, 435, 641--658, (2011).
		
		\bibitem{kokiopoulou2004computing}
		E. Kokiopoulou, C. Bekas, E. Gallopoulos, Computing smallest singular triplets with 
		implicitly restarted Lanczos bidiagonalization, Applied Numerical Mathematics, 49, 39--61
		(2004).
		
		\bibitem{kolda2009tensor}
		T. Kolda, B. W. Bader, Tensor decompositions and applications, SIAM Review, 51, 455--500
		(2009).
		
		
		\bibitem{georg}
		A.V. Nefian, Georgia Tech Face Database, Available online: 
		\url{http://www.anefian.com/research/face_reco.htm}.
		
		\bibitem{paige1995approximate}
		C. C. Paige, B. N. Parlett, H. A. Van der Vorst, Approximate solutions and eigenvalue 
		bounds from Krylov subspaces, Numerical Linear Algebra with Applications, 2, 115--133 
		(1995).
		
		\bibitem{paigesaunders}
		C. C. Paige, M. A. Saunders, LSQR: An algorithm for sparse linear equations and sparse 
		least squares, Transactions on Mathematical Software, 8, 43--71 (1982). 
		
		\bibitem{classifi3}
		P. K. Pandey, Y. Singh, S. Tripathi, Image processing using principle component 
		analysis, International Journal of Computer Applications, 15, 37--40 (2011).
		
		\bibitem{RU21}
		L. Reichel, U. O. Ugwu, Tensor Krylov subspace methods with an invertible linear 
		transform product applied to image processing, Applied Numerical Mathematics, 166, 
		186--207 (2021).
		
		\bibitem{RU22}
		L. Reichel and U. O. Ugwu, Tensor Arnoldi-Tikhonov and GMRES-type methods for ill-posed 
		problems with a t-product structure, Journal of Scientific Computing, 90, Art. 59 (2022).
		
		\bibitem{RU2022}
		L. Reichel, U. O. Ugwu, The tensor Golub-Kahan-Tikhonov method applied to the solution 
		of ill-posed problems with a t-product structure, Numerical Linear Algebra with 
		Applications, 29, e2412 (2022).
		
		\bibitem{Rojo}
		H. Rojo, O. Rojo, Some results on symmetric circulant matrices and on symmetric 
		centrosymmetric matrices, Linear Algebra and Its Applications, 391, 211-233 (2004).
		
		
		\bibitem{simon2000low}
		H. D. Simon, H. Zha, Low-rank matrix approximation using the Lanczos bidiagonalization 
		process with applications, SIAM Journal on Scientific Computing, 21, 2257--2274 (2000).
		
		\bibitem{Sorensen1992implicit}
		D. C. Sorensen, Implicit application of polynomial filters in a $k$-step Arnoldi method, 
		SIAM Journal on Matrix Analysis and Applications, 13, 357--385 (1992).
		
		\bibitem{stathopoulos1998restarting}
		A. Stathopoulos, Y. Saad, Restarting techniques for the (Jacobi-) Davidson symmetric 
		eigenvalue methods, Electronic Transactions on Numerical Analysis, 7, 163--181 (1998).
	\end{thebibliography}
\end{document}